\documentclass[pdflatex,sn-mathphys]{sn-jnl}% Math and Physical Sciences Reference Style
\usepackage{graphicx}
\usepackage{subfigure}
\usepackage[justification=centering]{caption}
\jyear{2021}%

\theoremstyle{thmstyleone}%
\newtheorem{theorem}{Theorem}
\newtheorem{corollary}{Corollary}
 
\newtheorem{proposition}[theorem]{Proposition}% 

\theoremstyle{thmstyletwo}%
\newtheorem{example}{Example}%
\newtheorem{remark}{Remark}%

\theoremstyle{thmstylethree}%
\newtheorem{definition}{Definition}%

\begin{document}

\title[The upper capacity topological entropy of free semigroup actions]{The upper capacity topological entropy of free semigroup actions for certain non-compact sets, \uppercase\expandafter{\romannumeral2}}

\author[1]{ \sur{Yanjie Tang }}\email{yjtang1994@gmail.com}
\author[2]{\sur{Xiaojiang Ye}}\email{yexiaojiang12@163.com}
\author*[3]{\sur{Dongkui Ma}}\email{dkma@scut.edu.cn}

\affil[1]{\orgdiv{School of Mathematics}, \orgname{South China University of Technology}, \city{Guangzhou}, \postcode{510641},  \country{P.R. China}}
\affil[2]{\orgdiv{School of Mathematics}, \orgname{South China University of Technology}, \city{Guangzhou}, \postcode{510641},  \country{P.R. China}}
\affil*[3]{\orgdiv{School of Mathematics}, \orgname{South China University of Technology}, \city{Guangzhou}, \postcode{510641},  \country{P.R. China}}


\abstract{
	This paper's major purpose is to continue the work of Zhu and Ma \cite{MR4200965}.  To begin, the $\mathbf{g}$-almost product property, more general irregular and regular sets,  and some new notions of the Banach upper density recurrent points and transitive points of free semigroup actions are introduced. Furthermore, under the $\mathbf{g}$-almost product property and other conditions, we coordinate the Banach upper recurrence, transitivity with (ir)regularity, and obtain lots of generalized multifractal analyses for general observable functions of free semigroup actions. Finally, statistical $\omega$-limit sets are used to consider the upper capacity topological entropy of the sets of Banach upper recurrent points and transitive points of free semigroup actions, respectively. Our analysis generalizes the results obtained by Huang, Tian and Wang \cite{MR3963890}, Pfister and Sullivan \cite{MR2322186}.
}

\keywords{Free semigroup actions; Upper capacity topological entropy; Multifractal analysis; $\mathbf{g}$-almost product property; Banach upper recurrence}

\pacs[MSC Classification]{37B40, 37B20, 37C45, 37B05}

\maketitle

\section{Introduction}\label{sec1}
A dynamical system $(X,d,f)$ means always that $(X,d)$ is a compact metric space and $f$ is a continuous self map on $X$. One of the fundamental issues in dynamical systems is how points with similar asymptotic behavior influence or determine the system's complexity. Topological entropy is a term used to describe the dynamical complexity of a dynamical system.  {Using a similar defining way as the Hausdorff dimension, Bowen  \cite{MR338317} and Pesin \cite{MR1489237} extended the concept of topological entropy to non-compact sets. On the basis of this theory, researchers are paying more attention to study Hausdorff dimension or topological entropy for non-compact sets, for example, multifractal spectra and saturated sets of dynamical systems, see, \cite{MR2322186,MR3436391,MR1971209, MR1439805, MR2931333,  MR1759398, MR3963890, MR3833343, MR2158401}. In particular, recent work about non-compact sets with full entropy has attracted greater attention than before, see \cite{ MR4200965, MR3963890, MR3833343, MR3436391,MR2158401}.}
As a result, the size of the topological entropy of the ‘periodic-like’ recurrent level sets  is always in focus, see \cite{MR3963890,DongandTian1,DongandTian2,MR3436391,MR4200965}.
%%人们关心这种集合的大小

Most of different recurrent level sets are considered. The concepts of periodic point, recurrent point, almost periodic point and non-wandering point are described, see \cite{MR648108}, the concepts of (quai)weakly almost periodic point (sometimes called (upper)lower density recurrent point) can be see in \cite{MR1223083,MR2039054,1993Measure,1995Level}. Moreover, Zhou \cite{1995Level} firstly linked the recurrent frequency with the support of invariant measures which are the limit points of the empirical measure, which provided a tremendous benefit to researchers.

Birkhoff ergodic average \cite{MR1439805,MR1489237}(or Lyapunov exponents \cite{MR2645746}) is a well-known method to distinguish the asymptotic behavior of a dynamical orbit. Let $\varphi: X\to\mathbb{R}$ be a continuous function and $a$ a real number.  Consider a level set of Birkhoff averages
$$
R_\varphi(a):=\left \{x\in X: \lim_{n\to \infty}\frac{1}{n}\sum_{j=0}^{n-1}\varphi\left (f^j(x)\right )=a\right \}.
$$
These $R_\varphi(a)$ form a multifractal decomposition and the function 
$$
a\to h_{R_\varphi(a)}(f)
$$
is a entropy spectrum, where $h_{R_\varphi(a)}(f)$ denotes the topological entropy of $R_\varphi(a)$ defined by \cite{MR338317}. Takens and Verbitskiy \cite{MR1971209} proved that transformations satisfying the specification has the following variational principle
\begin{equation}\label{1.1}
	h_{R_\varphi(a)}\left(f\right)=\sup \left\{ h_{\mu}(f): \mu\text{ is invariant and }\int \varphi d \mu=a\right\}
\end{equation}
where $h_\mu(f)$ is the measure-theoretic (or Kolmogorov-Sinai) entropy of the invariant measure $\mu$. In \cite{MR2322186}, the formula (\ref{1.1}) holds  under the conditions of the $\mathbf{g}$-almost product property which was introduced by Pfister and Sullivan. Consider $\varphi$-irregular set as follows: 
$$
I_\varphi:=\left \{x\in X:\lim_{n\to\infty}\frac{1}{n}\sum_{j=0}^{n-1}\varphi \left (f^j(x)\right )\text{ diverges}\right \}.
$$
The irregular set is not detectable from the standpoint of invariant measures, according to Birkhoff's ergodic theorem. However, in systems with certain properties, the irregular set may carry full topological entropy, see \cite{MR2931333,MR3833343,MR2158401,MR775933,MR1759398} and references therein. 

Based on these results, Tian \cite{MR3436391}  distinguished various periodic-like recurrences and discovered that they all carry full topological entropy, as do their gap-sets, under the $\mathbf{g}$-almost product property and other conditions. Furthermore, the author  \cite{MR3436391} combined the periodic-like recurrences with (ir)regularity to obtain a large number of generalized multifractal analyses for all continuous observable functions. Later on, Huang, Tian and Wang \cite{MR3963890} introduced an abstract version of multifractal analysis for possible applicability to more general function. Let $\alpha :\mathcal{M}(X,f)\to\mathbb{R}$ be a continuous function where $\mathcal{M}(X,f)$ denotes the set of all $f$-invariant probability measures on $X$. Define more general regular and irregular sets as follows
$$
R_\alpha:=\left \{x\in X : \inf_{\mu\in M_x}\alpha (\mu)=  \sup_{\mu\in M_x}\alpha (\mu)\right \},
$$
$$
I_\alpha:=\left \{x\in X : \inf_{\mu\in M_x}\alpha (\mu)<  \sup_{\mu\in M_x}\alpha (\mu)\right \},
$$
respectively, where $M_x$ stands for the set of all limit points of the empirical measure. Furthermore, the authors \cite{MR3963890} established an abstract framework on the combination of {Banach upper recurrence,} transitivity and multifractal analysis of general observable functions. Learned from \cite{MR1601486} for maps and \cite{MR1716564} for flows, Dong and Tian \cite{DongandTian1} defined the statistical $\omega$-limit sets to describe different statistical structure of dynamical orbits using concepts of density, and considered multifractal analysis on various non-recurrence and Birkhoff ergodic averages. Furthermore,  the authors \cite{MR3963890} used statistical $\omega$-limit sets to obtain the results of topological entropy on the refined orbit distribution of
{Banach upper recurrence.}

On the other hand, Ghys et al \cite{MR926526} proposed a definition of topological entropy for finitely generated pseudo-groups of continuous maps.
%{
%it is needed by some other disciplines, such as physics, to allow the system that describes the real events to readjust over time to account for the inevitable experimental errors in \cite{MR2808288}. Some dynamic system theories, on the other hand, are closely related to it, such as the case of a foliation on a manifold and a pseudo-group of holonomy maps. 
%The geometric entropy of finitely generated pseudogroup has been introduced \cite{MR926526} and shown to be a useful tool for studying the topology and dynamics of foliated manifolds. } 
Later on, the topological entropy of free semigroup actions on a compact metric space defined by Bi\'{s} \cite{MR2083436} and Bufetov \cite{MR1681003}, respectively. Rodrigues and Varandas \cite{MR3503951} and Lin et al \cite{MR3774838} extended the work of Bufetov \cite{MR1681003} from various perspectives. Similar to the methods of Bowen \cite{MR338317} and Pesin \cite{MR1489237}, Ju et al \cite{MR3918203} introduced topological entropy and upper capacity topological entropy of free semigroup actions on non-compact sets which extended the results obtained by \cite{MR1681003,MR338317,MR1489237}.  Zhu and Ma \cite{MR4200965} investigated the upper capacity topological entropy of free semigroup actions for certain non-compact sets. More relevant results are obtained, see \cite{MR3503951,MR3828742,MR3784991,MR3774838,MR3918203,MR4200965,MR3592853,MR1767945}. 

{
	The above results raise the question of whether similar sets exist in dynamical system of free semigroup actions. Further, one can ask if the sets have full topological entropy or full upper capacity topological entropy, and if  an abstract framework of general observable functions can be established. To answer the above questions, we introduce different asymptotic behavior of points and more general irregular set. Our results show that various subsets characterized by distinct asymptotic behavior may carry full upper capacity topological entropy of free semigroup actions. Our analysis is to continue the work of \cite{MR4200965} and generalize the results obtained by \cite{MR3963890} and \cite{MR2322186}.}

%focusing on the upper capacity topological entropy with respect to the set of  (Banach) upper recurrent points of free semigroup actions,  and to establish an abstract framework on the combination of Banach upper recurrence, transitivity and multifractal analysis of general observable functions of free semigroup actions.

This paper is organized as follows. In Sec. \ref{main results}, we give our main results. In Sec. \ref{Section-Preliminaries-2}, we give some preliminaries.  
{
In Sec. \ref{3}, we introduce some concepts of transitive points, quasiregular points, upper recurrent points and Banach upper recurrent points  with respect to a certain orbit of free semigroup actions.
On the other hand, the g-almost product property of free semigroup actions, 
which is weaker than the specification property of free semigroup actions, is introduced. 
}
In Sec. \ref{Irregular and regular set}, the upper capacity topological entropy of free semigroup actions on the more general irregular and regular sets  is considered, and some examples satisfying the assumptions of our main results are described. In Sec. \ref{BR and QW}, we give the proofs of our main results.

\section{Statement of main results}\label{main results}
Let $(X,d)$ be a compact metric space and $G$ the free semigroup generated by $m$ generators $f_0,\cdots,f_{m-1}$ which are continuous maps on $X$. Let $\Sigma^+_m$ denote the one-side symbol space generated by the digits $\{0,1,\cdots,m-1\}$. Recall that the skew product transformation {is given as follows}:
$$
F:\Sigma^+ _m \times X\to\Sigma^+ _m \times X,\:\, (\iota,x)\mapsto \big(\sigma(\iota),f_{i_0}(x)\big),
$$
where $\iota=(i_0, i_1,\cdots)\in\Sigma_{m}^+$, $x\in X$ and $\sigma$ is the shift map of $\Sigma^+ _m $. Let $\mathcal{M}(\Sigma_{m}^+\times X,F)$ denote the set of invariant measures of the skew product $F$.  Denote by $\mathrm{Tran}(\iota,G)$, $\mathrm{QR}(\iota,G)$, $\mathrm{QW}(\iota,G)$ and $\mathrm{BR}(\iota,G)$ the sets of the transitive points, the quasiregular points, the upper recurrent points and the Banach upper recurrent points  of free semigroup action $G$ with respect to $\iota\in\Sigma_{m}^+$, respectively (see Sec. \ref{3}).

Let $\alpha: \mathcal{M}(\Sigma_{m}^+\times X, F)\to\mathbb{R}$ be a continuous function. Given $(\iota, x) \in \Sigma_{m}^+\times X$, let $M_{(\iota,x)}(F)$ be the set of all limit points of $\{\frac{1}{n}\sum_{j=0}^{n-1}\delta_{F^j(\iota,x)}\}$ in weak$^*$ topology.  Let $L_\alpha:=[\inf_{\nu\in M_{(\iota,x)}(F)}\alpha(\nu),\, \sup_{\nu\in M_{(\iota,x)}(F)}\alpha(\nu)]$ and $\mathrm{Int}(L_\alpha)$ denote the interior of interval $L_\alpha$. 
{
	  For possibly applicability to more general functions, \cite{MR3963890} defined three conditions for $\alpha$ to introduce an abstract version of multifractal analysis. To understand our main results, we list the three conditions}(see \cite{MR3963890} for details):
		\begin{itemize}
			\item [A.1] For any $\mu, \nu \in \mathcal{M}(\Sigma_{m}^+\times X, F)$, $\beta(\theta):=\alpha(\theta \mu+(1-\theta) \nu)$ is strictly monotonic on $[0,1]$ when $\alpha(\mu) \neq \alpha(\nu)$.
			\item [A.2] For any $\mu, \nu \in \mathcal{M}(\Sigma_{m}^+\times X, F)$, $\beta(\theta):=\alpha(\theta \mu+(1-\theta) \nu)$ is constant on $[0,1]$ when $\alpha(\mu)=\alpha(\nu)$.
			\item [A.3] For any $\mu, \nu \in \mathcal{M}(\Sigma_{m}^+\times X, F)$, $\beta(\theta):=\alpha(\theta \mu+(1-\theta) \nu)$ is not constant over any subinterval of $[0,1]$ when $\alpha(\mu) \neq \alpha(\nu)$.
		\end{itemize}
	
Given $\iota\in\Sigma_{m}^+$, we introduce the general regular set and irregular set with respect to $\iota$ as follows:
	$$
	R_\alpha(\iota,G):=\left \{x\in X: \inf_{\nu\in M_{(\iota,x)}(F)}\alpha(\nu)= \sup_{\nu\in M_{(\iota,x)}(F)}\alpha(\nu)\right \},
	$$
	$$
	I_\alpha(\iota,G):=\left \{x\in X: \inf_{\nu\in M_{(\iota,x)}(F)}\alpha(\nu)< \sup_{\nu\in M_{(\iota,x)}(F)}\alpha(\nu)\right \}.
	$$

Let $C_{(\iota,x)}=\overline{\cup_{\nu \in M_{(\iota,x)}(F)} S_{\nu}}$, where $S_\nu$ denotes the support of measure $\nu$.
Let $\mathrm{BR}^{\#}(\iota,G):=\mathrm{BR}(\iota,G) \backslash \mathrm{QW} (\iota,G)$. Fixed $\iota\in \Sigma_{m}^+$, the following sets are introduced to more precisely characterize the recurrence of the orbit,
% (see Fig. \ref{figure-1} and \ref{figure-2}),
$$
\begin{aligned}
	W(\iota,G)&:=\left\{x\in X :S_{\nu}=C_{(\iota,x)} \text { for every } \nu \in M_{(\iota,x)}(F)\right\}, \\
	V(\iota,G)&:=\left\{x\in X: \exists \nu \in M_{(\iota,x)}(F) \text { such that } S_{\nu}=C_{(\iota, x)}\right\}, \\
	S(\iota,G)&:=\left\{x\in X: \cap_{\nu \in M_{(\iota,x)}(F)} S_{\nu} \neq \emptyset\right\} .
\end{aligned}
$$

 More specifically, $\mathrm{BR}^{\#}(\iota,G)$ is subdivided into the following several levels with different asymptotic behaviour:
 $$
 \begin{aligned}
 	&\mathrm{BR}_{1}(\iota,G):=\mathrm{BR}^{\#}(\iota,G)\cap W(\iota,G), \\
 	&\mathrm{BR}_{2}(\iota,G):=\mathrm{BR}^{\#}(\iota,G)\cap(V(\iota,G) \cap S(\iota,G)), \\
 	&\mathrm{BR}_{3}(\iota,G):=\mathrm{BR}^{\#}(\iota,G)\cap V(\iota,G),\\
 	&\mathrm{BR}_{4}(\iota,G):=\mathrm{BR}^{\#}(\iota,G)\cap (V(\iota,G) \cup S(\iota,G)),\\
 	&\mathrm{BR}_{5}(\iota, G):=\mathrm{BR}^{\#}(\iota,G).
 \end{aligned}
 $$
 Then $\mathrm{BR}_{1}(\iota,G) \subseteq \mathrm{BR}_{2}(\iota,G) \subseteq \mathrm{BR}_{3}(\iota,G) \subseteq \mathrm{BR}_{4}(\iota,G) \subseteq \mathrm{BR}_{5}(\iota,G)$.
 
 Analogously, $\mathrm{QW}(\iota,G)$ is also subdivided into the following several levels with varying asymptotic behaviour:
 $$
 \begin{aligned}
 	&\mathrm{QW}_{1}(\iota,G):=\mathrm{QW}(\iota,G)\cap W(\iota,G), \\
 	&\mathrm{QW}_{2}(\iota,G):=\mathrm{QW}(\iota,G)\cap (V(\iota,G) \cap S(\iota,G)), \\
 	&\mathrm{QW}_{3}(\iota,G):=\mathrm{QW}(\iota,G)\cap V(\iota,G),\\
 	&\mathrm{QW}_{4}(\iota,G):=\mathrm{QW}(\iota,G)\cap (V(\iota,G) \cup S(\iota,G)),\\
 	&\mathrm{QW}_{5}(\iota, G):=\mathrm{QW}(\iota,G).
 \end{aligned}
 $$
 Note that $\mathrm{QW}_{1}(\iota,G) \subseteq \mathrm{QW}_{2}(\iota,G) \subseteq \mathrm{QW}_{3}(\iota,G) \subseteq \mathrm{QW}_{4}(\iota,G) \subseteq \mathrm{QW}_{5}(\iota,G)$.

Let $Z_0(\iota,G),Z_1(\iota,G),\cdots,Z_k(\iota,G)$ and $Y(\iota,G)$ be the subsets of $X$ with respect to $\iota\in\Sigma_{m}^+$. Consider the sets, for $j=1,\cdots,k$,
$$
M_j(Y):=\cup_{\iota\in\Sigma_{m}^+}\big(Y(\iota,G)\cap(Z_j(\iota,G)\setminus Z_{j-1}(\iota,G))\big),
$$
then we say that the sets $M_1(Y), M_2(Y),\cdots,M_k(Y)$ is the unions of gaps of $\{Z_0(\iota,G),Z_1(\iota,G),\cdots,Z_k(\iota,G)\}$  with respect to $Y(\iota,G)$ for all $\iota\in\Sigma_{m}^+$.

Now we start to state our main theorems.
	\begin{theorem}\label{entropy of BR-1}
		Suppose that $G$ has the $\mathbf{g}$-almost product property, there exists a $\mathbb{P}$-stationary measure (see Sec. \ref{Stationary measure}) on $X$ with full support where $\mathbb{P}$ is a Bernoulli measure on $\Sigma_{m}^+$.  Let $\alpha: \mathcal{M}(\Sigma_{m}^+\times X, F)\to\mathbb{R}$ be a continuous function.
		\begin{itemize}
			\item [(1)] If $\alpha$ satisfies A.3 and $\mathrm{Int}(L_\alpha)\neq\emptyset$, then the unions of gaps of
			$$
			\left \{\mathrm{QR}(\iota, G), \mathrm{BR}_1(\iota,G), \mathrm{BR}_2(\iota,G),\mathrm{BR}_3(\iota,G),\mathrm{BR}_4(\iota,G),\mathrm{BR}_5(\iota,G)\right \}
			$$
			with respect to $I_\alpha (\iota,G)\cap\mathrm{Tran}(\iota,G)$ for all $\iota\in\Sigma_{m}^+$ have full upper capacity topological entropy of free semigroup action $G$.
			\item [(2)] If the skew product $F$ is not uniquely ergodic, then the unions of gaps of
			$$
			\left \{\emptyset, \mathrm{QR}(\iota, G)\cap \mathrm{BR}_1(\iota,G),\mathrm{BR}_1(\iota,G), \mathrm{BR}_2(\iota,G),\mathrm{BR}_3(\iota,G),\mathrm{BR}_4(\iota,G),\mathrm{BR}_5(\iota,G)\right \}
			$$
			with respect to $\mathrm{Tran}(\iota,G)$ for all $\iota\in\Sigma_{m}^+$ have full upper capacity topological entropy of free semigroup action $G$.
			\item [(3)] If the skew product $F$ is not uniquely ergodic and $\alpha$ satisfies A.1 and A.2, then the unions of gaps of
			$$
			\left \{\emptyset, \mathrm{QR}(\iota, G)\cap \mathrm{BR}_1(\iota,G),\mathrm{BR}_1(\iota,G), \mathrm{BR}_2(\iota,G),\mathrm{BR}_3(\iota,G),\mathrm{BR}_4(\iota,G),\mathrm{BR}_5(\iota,G)\right \}
			$$
			with respect to $R_\alpha(\iota,G)\cap\mathrm{Tran}(\iota,G)$ for all $\iota\in\Sigma_{m}^+$ have full upper capacity topological entropy of free semigroup action $G$.
		\end{itemize}
	\end{theorem}
	\begin{theorem}\label{entropy of QW-1}
		Suppose that $G$ has the $\mathbf{g}$-almost product property and positively expansive, there exists a $\mathbb{P}$-stationary measure with full support on $X$ where $\mathbb{P}$ is a Bernoulli measure on $\Sigma_{m}^+$. Let $\varphi: X\to\mathbb{R}$ be a continuous function. If the skew product $F$ is not uniquely ergodic, then the unions of gaps of
		$$
		\left \{\emptyset, \mathrm{QR}(\iota,G)\cap\mathrm{QW}_1(\iota,G),\mathrm{QW}_1(\iota,G),\mathrm{QW}_2(\iota,G),\mathrm{QW}_3(\iota,G),\mathrm{QW}_4(\iota,G),\mathrm{QW}_5(\iota,G)\right \}
		$$
		with respect to 
		$\mathrm{Tran}(\iota,G)$, $R_\alpha(\iota,G)\cap\mathrm{Tran}(\iota,G)$ for all $\iota\in\Sigma_{m}^+$ have full upper capacity topological entropy of free semigroup action $G$, respectively.  
		If {$I_\alpha(\iota,G)$} is non-empty for some $\iota\in\Sigma_{m}^+$, similar arguments hold with respect to $I_\alpha (\iota,G)\cap\mathrm{Tran}(\iota,G)$.
	\end{theorem}

For $S\subseteq\mathbb{N}$, let $\overline{d}(S) $, $\underline{d}(S)$, $B^{*}(S)$ and $B_*(S)$ denote the upper density, lower density, Banach upper density and Banach lower density of $S$, respectively. Given $(\iota, x)\in\Sigma_{m}^+\times X$ and $\xi\in\{\,\overline{d}, \underline{d},B^{*}, B_{*} \}$, denote by $\omega_\xi\left ((\iota,x),F\right )$ the $\xi$-$\omega$-limit set of $(\iota, x)$. The notions will be given in more
detail later in Section \ref{Section-Preliminaries-2-1} (also see \cite{MR3963890, DongandTian1,DongandTian2}). If $\omega_{B_*}\left ((\iota,x),F\right )=\emptyset$ and $\omega_{B^*}\left ((\iota,x),F\right )=\omega\left ((\iota,x),F\right )$, then from \cite{DongandTian1} one has that $(\iota,x)$ satisfies only one of the following six cases:
\begin{enumerate}
	\item [(1)]
	$\emptyset=\omega_{B_{*}}\left ((\iota,x),F\right ) \subsetneq \omega_{\underline{d}}\left ((\iota,x),F\right )=\omega_{\overline{d}}\left ((\iota,x),F\right )=\omega_{B^{*}}\left ((\iota,x),F\right )=\omega\left ((\iota,x),F\right )$;
	\item [(2)]
	$\emptyset=\omega_{B_{*}}\left ((\iota,x),F\right ) \subsetneq \omega_{\underline{d}}\left ((\iota,x),F\right )=\omega_{\overline{d}}\left ((\iota,x),F\right ) \subsetneq \omega_{B^{*}}\left ((\iota,x),F\right )=\omega\left ((\iota,x),F\right )$;
	\item [(3)]
	$\emptyset=\omega_{B_{*}}\left ((\iota,x),F\right )=\omega_{\underline{d}}\left ((\iota,x),F\right ) \subsetneq \omega_{\overline{d}}\left ((\iota,x),F\right )=\omega_{B^{*}}\left ((\iota,x),F\right )=\omega\left ((\iota,x),F\right )$;
	\item [(4)]
	$\emptyset=\omega_{B_{*}}\left ((\iota,x),F\right ) \subsetneq \omega_{\underline{d}}\left ((\iota,x),F\right ) \subsetneq \omega_{\overline{d}}\left ((\iota,x),F\right )=\omega_{B^{*}}\left ((\iota,x),F\right )=\omega\left ((\iota,x),F\right )$;
	\item [(5)]
	$\emptyset=\omega_{B_{*}}\left ((\iota,x),F\right )=\omega_{\underline{d}}\left ((\iota,x),F\right ) \subsetneq \omega_{\overline{d}}\left ((\iota,x),F\right ) \subsetneq \omega_{B^{*}}\left ((\iota,x),F\right )=\omega\left ((\iota,x),F\right )$;
	\item [(6)]
	$\emptyset=\omega_{B_{*}}\left ((\iota,x),F\right ) \subsetneq \omega_{\underline{d}}\left ((\iota,x),F\right ) \subsetneq \omega_{\overline{d}}\left ((\iota,x),F\right ) \subsetneq \omega_{B^{*}}\left ((\iota,x),F\right )=\omega\left ((\iota,x),F\right )$.
\end{enumerate}
Consider the two sets as follows:
$$
T_j(\iota,G):=\left \{x\in \mathrm{Tran} (\iota,G):(\iota,x) \text{ satisfies Case } (j)\right \},
$$
$$
B_j(\iota,G):=\left \{x\in \mathrm{BR} (\iota,G):(\iota,x) \text{ satisfies Case } (j)\right \},
$$
where $j=1,\cdots,6$. Let
$$
T_j(G):=\cup_{\iota\in\Sigma_{m}^+}T_j(\iota,G),\quad B_j(G):=\cup_{\iota\in\Sigma_{m}^+}B_j(\iota,G).
$$

\begin{theorem}\label{entropy of omega limit set}
	Suppose that $G$ has the $\mathbf{g}$-almost product property, there exists a $\mathbb{P}$-stationary measure with full support on $X$ where $\mathbb{P}$ is a Bernoulli measure on $\Sigma_{m}^+$. If the skew product $F$ is not uniquely ergodic, then $T_j(G)\neq\emptyset$ and $B_j(G)\neq\emptyset$.
	Moreover, they all have full capacity topological entropy of free semigroup action $G$, that is,
	$$
	\overline{Ch}_{T_j(G)}(G)=\overline{Ch}_X(G)=h(G),
	$$
	$$
	\overline{Ch}_{B_j(G)}(G)=\overline{Ch}_X(G)=h(G),
	$$
	for all $j=1,\cdots,6$, where $\overline{Ch}_Z(G)$ denotes the upper capacity topological entropy on any subset $Z\subseteq X$ in the sense of \cite{MR3918203}, $h(G)$ denotes the topological entropy in the sense of Bufetov \cite{MR1681003}. If $Z=X$, we have $\overline{Ch}_{X}(G)=h(G)$ from Remark 5.1 of \cite{MR3918203}.
\end{theorem}

\section{Preliminaries}\label{Section-Preliminaries-2}

\subsection{Some notions} \label{Section-Preliminaries-2-1}
Let $(X,d)$ be a compact metric space and $f$ be a continuous map on $X$.
For $S\subseteq \mathbb{N}$, the upper density and the Banach upper density of $S$ are defined by
$$
\overline{d}(S):=\limsup_{n \rightarrow \infty} \frac{\sharp\left \{S \cap\{0,1, \cdots, n-1\}\right \}}{n},\quad B^{*}(S):=\limsup _{\sharp I \rightarrow \infty} \frac{\sharp \left \{S \cap I\right \}}{\sharp I},
$$
respectively, where $\sharp Y$ denotes the cardinality of the set $Y$ and $I \subseteq \mathbb{N}$ is taken from finite continuous integer intervals. Similarly, one can define 
the lower density  and the Banach lower density of $S$, denoted as $\underline{d}(S)$ and $B_{*}(S)$, respectively. Let $U\subset X$ be a nonempty open set and $x\in X$. Define the set of visiting time,
$$
N(x,U):=\left \{ n\ge 0: f^n(x)\in U\right \}.
$$

Recall that 
{a point $x\in X$ is called to be Banach upper recurrent,}
if for any $\varepsilon>0$, the set of visiting time $N(x,B(x,\varepsilon))$ has a positive Banach upper density where $B(x,\varepsilon)$ denotes the ball centered at $x$ with radius $\varepsilon$. Similarly, one can call 
{a point $x\in X$ upper recurrent,}
if for any $\varepsilon>0$, the set of visiting time $N(x,B(x,\varepsilon))$ has a positive upper density.  Let us denote by $\mathrm{BR}(f)$ and $\mathrm{QW}(f)$ the sets of the Banach upper recurrent points and the upper recurrent points of $f$, respectively. It is immediate that 
$$
\mathrm{QW}(f)\subseteq \mathrm{BR}(f).
$$
A {point} $x\in X$ is called transitive {if its orbit $\{x,f(x),f^2(x),\cdots\}$} is dense in $X$. Let us denote by $\mathrm{Tran}(f)$ the set of transitive points of $f$.

{We recall} that several concepts were introduced in \cite{MR3963890}. For $x \in X$ and $\xi\in\{\,\overline{d}, \underline{d}, B_{*}, B^{*} \}$, a point $y \in X$ is called $x$-$\xi$-accessible, if for any $\varepsilon>0$, the set of visiting time $N\left(x, B(y,\varepsilon)\right)$ has positive density with respect to $\xi$. Let
$$
\omega_{\xi}(x):=\{y \in X: y \text { is } x\text{-}\xi\text{-accessible }\}.
$$
For convenience, it is called the $\xi$-$\omega$-limit set of $x$.

The set of invariant measures under $f$ will be denote by $\mathcal{M}(X,f)$. For $\mu\in\mathcal{M}(X,f)$, a point $x \in X$ is $\mu$-generic if
	$$
	\lim _{n \rightarrow \infty} \frac{1}{n} \sum_{j=0}^{n-1} \delta_{f^{j}(x)}=\mu
	$$
where $\delta_{y}$ denotes the Dirac measure on $y$. We will use $G_\mu(f)$ to denote the set of $\mu$-generic points. Let $\mathrm{QR}(f):=\bigcup_{\mu\in\mathcal{M}(X,f)}G_\mu(f)$. The points in $\mathrm{QR}(f)$ are called quasiregular points of $f$.

\subsection{The topological entropy and others concepts of free semigroup actions} \label{the concepts of free semigroup actions}
In this paper, we use the topological entropy and upper capacity topological entropy of free semigroup actions {defined by \cite{MR1681003} and \cite{MR3918203}, respectively.} Let $(X,d)$ be a compact metric space and $G$ the free semigroup action on $X$ generated by $f_0,\cdots,f_{m-1}$.  For convenience, we first recall the notion of words.

Let $F_m^+$ be the set of all finite words of symbols $0,1,\cdots,m-1$.  For any $w\in F_m^+$, $\lvert w\rvert$ stands for the length of $w$, that is, the number of symbols in $w$. Obviously, $F^+_m$ with respect to the law of composition is a free semigroup with $m$ generators. We write $w'\leq w$ if there exists a word $w''\in F^+_m$ such that $w=w''w'$. Remark that $\emptyset\in F_m^+$ and $\emptyset\leq w$. For $w=i_0i_1\cdots i_k\in F^+_m$, denote $\overline{w}=i_k\cdots i_1i_0$.

Denote by $\Sigma^+_m$ the set of all one-side infinite sequences of symbols $\{0,1,\cdots,m-1\}$, that is, 

$$\Sigma^+_m=\left\{\iota=(i_0,i_1,\cdots)\,  : \,i_k\in \{0,1,\cdots,m-1\},\: k\in\mathbb{N}\right\}.
$$
{The metric on  $\Sigma^+_m$ is given by 
	$$d'(\iota,\iota'):=2^{-j},\quad j=\inf\{n\, :  \,i_n\neq i'_n\}.
	$$}It is easy to check that $\Sigma^+_m$ is compact with respect to this metric. The shift $\sigma:\Sigma^+_m\to \Sigma^+_m $ is given by the formula, for each $\iota=(i_0,i_1,\cdots)\in\Sigma^+_m$,
$$
\sigma(\iota)=(i_1,i_2,\cdots).
$$

Suppose that $\iota\in\Sigma^+_m$, and $a,b\in \mathbb{N}$ with $a\leq b$. We write $\iota\lvert_{[a,b]}=w$ if $w=i_ai_{a+1}\cdots i_b$.

To each $w\in F^+_m$, {$w=i_{0}\cdots i_{k-1}$, let us write $f_w=f_{i_{0}}\circ\cdots \circ f_{i_{k-1}}$} if $k>0$,  and $f_w=\mathrm{Id}$ if $k=0$, where $\mathrm{Id}$ is the {identity map}. Obviously, $f_{ww'}=f_wf_{w'}$.  

For $w\in F^+_m$, we assign a metric $d_w$ on $X$ by setting
$$
d_w(x_1,x_2)=\max_{w^{\prime} \leq \overline{w}}d\left (f_{w^{\prime}}(x_1),f_{w^{\prime}}(x_2)\right ).
$$
Given a number $\delta>0$ and a point $x \in X$, define the $(w, \delta)$-Bowen ball at $x$ by
$$
B_{w}(x, \delta):=\left\{y \in X:{d_w\left(x, y\right) < \delta}\right\} .
$$

Recall that the positively expansive of the free semigroup actions means that if there exists $\delta>0$, such that any $x,y\in X$ with $x\neq y$, for any $\iota\in\Sigma_{m}^+$ there exists $n\ge 1$ satisfying $d\left (f_{\overline{\iota \lvert_{[0,n-1]}}}(x),f_{\overline{\iota \lvert_{[0,n-1]}}}(y)\right )\ge \delta$, which was introduced by Zhu and Ma \cite{MR4200965}.

The specification property of free semigroup actions was introduced by Rodrigues and Varandas \cite{MR3503951}. We say that $G$ has the specification property if for any $\varepsilon>0$, there exists $\mathfrak{p}(\varepsilon)>0$, such that for any $k>0$, any points $x_{1},\cdots, x_{k} \in X$, any positive integers $n_{1}, \cdots, n_{k}$,  any $p_{1},\cdots, p_{k} \geq \mathfrak{p}(\varepsilon)$, any $w_{(p_{j})} \in F_{m}^{+}$ with $\lvert w_{(p_{j})}\rvert =p_{j}$,  $w_{(n_{j})} \in F_{m}^{+}$ with $\lvert w_{(n_{j})}\rvert =n_{j}, 1 \leq j \leq k$, one has
$$
	B_{w_{(n_1)}}\left (x_1,\varepsilon\right )\cap\left (\bigcap_{j=2}^k {f^{-1}_{\overline{w_{(p_{j-1})}}\, \overline{w_{(n_{j-1})}}\cdots\overline{w_{(p_1)}}\, \overline{w_{(n_1)}}}}B_{w_{(n_j)}}\left (x_j,\varepsilon\right )\right )\neq\emptyset.
	$$If $m=1$,  the specification property of free semigroup actions coincides with the classical definition introduced by Bowen \cite{MR282372}.

{We recall the definition of  topological entropy} for free semigroup actions introduced by \cite{MR1681003}.
A subset $K$ of $X$ is called a $(w, \varepsilon, G)$-separated subset if, for any $x_{1}, x_{2} \in K$ with $x_{1} \neq x_{2}$, one has $d_{w}\left(x_{1}, x_{2}\right) \geq \varepsilon$. The maximum cardinality of a $(w, \varepsilon, G)$-separated subset of $X$ is denoted by $N(w, \varepsilon, G)$.
The topological entropy of free semigroup actions is defined by the formula
$$
h (G):=\lim _{\varepsilon \rightarrow 0} \limsup_{n\rightarrow\infty} \frac{1}{n} \log \left (\frac{1}{m^{n}} \sum_{\lvert w\rvert =n} N(w, \varepsilon, G)\right ).
$$
\begin{remark}
	If $m=1$,  this definition coincides with the topological entropy of a single map defined by \cite{MR175106}. For more information, see Chapter 7 of  \cite{MR648108}.
\end{remark}

{The dynamical systems given by free semigroup action}
have a strong connection with skew product which has been analyzed to obtain properties of free semigroup actions through fiber associated with the skew product (see for instance \cite{MR4200965, MR3784991}). Recall that the skew product transformation is given by as follows:
$$
F:\Sigma^+ _m \times X\to\Sigma^+ _m \times X,\:\, (\iota,x)\mapsto \big(\sigma(\iota),f_{i_0}(x)\big),
$$
where $\iota=(i_0, i_1,\cdots)$ and $\sigma$ is the shift map of $\Sigma^+ _m $. {The metric $D$ on} $\Sigma^+ _m \times X$ is given by the formula 
$$
D\left (\left (\iota,x\right ),\left (\iota',x'\right )\right ):=\max\left \{d'\left (\iota,\iota'\right ),d\left (x,x'\right )\right \}.
$$

\begin{theorem}
	\label{Topological entropy skew product}
	(\cite{MR1681003},Theorem 1) Topological entropy of the skew product transformation $F$ satisfies
	$$
	h(F)=\log m +h(G),
	$$
	where $h(F)$ denotes the topological entropy of $F$.
\end{theorem}

Now, let us recall the  topological entropy and upper capacity topological entropy of free semigroup actions for non-compact sets defined by \cite{MR3918203}.  Fixed $\delta>0$, we define the collection of subsets
$$
\mathcal{F}:=\left \{B_{w}(x, \delta): \,x\in X,\, w\in F^+_m, \,\lvert w\rvert =n \text{ and } n\in\mathbb{N}\right \}.
$$
Given subset $Z\subset X$,  we define,  for $\gamma\ge 0$, $N>0$ and $w\in F_m^+$ with $\lvert w\rvert =N$,
$$
{M}_{w}(Z, \gamma, \delta, N):=\inf _{\mathcal{G}_{w}}\left \{\sum_{B_{w'}(x, \delta) \in \mathcal{G}_{w}} \exp {\left (-\gamma \cdot\left(\lvert w'\rvert +1\right)\right )}\right\},
$$
where the infimum is taken over all finite or countable subcollections $\mathcal{G}_w\subseteq\mathcal{F}$ covering $Z$ (i.e. for any $B_{w'}(x,\delta)\in\mathcal{G}_w$, $\overline{w}\le \overline{w'}$ and $\bigcup_{B_{w'}(x,\delta)\in\mathcal{G}_w}B_{w'}(x,\delta)\supseteq Z$). Let
$$
{M}(Z, \gamma, \delta, N):=\frac{1}{m^{N}} \sum_{\lvert w\rvert =N} {M}_{w}(Z, \gamma, \delta, N).
$$
It is easy to verify that the function ${M}(Z, \gamma, \delta, N)$ is non-decreasing as $N$ increases. Therefore, there exists the limit
$$
{m}(Z, \gamma, \delta)=\lim _{N \rightarrow \infty} {M}(Z, \gamma, \delta, N).
$$

Furthermore, we can define
$$
{R}_{w}(Z, \gamma, \delta, N) :=\inf _{\mathcal{G}_{w}}\left \{\sum_{B_{w}(x, \delta) \in \mathcal{G}_{w}} \exp{\left ( -\gamma \cdot(N+1)\right )}\right \},
$$
where the infimum is taken over all finite or countable subcollections $\mathcal{G}_w\subseteq\mathcal{F}$ covering $Z$ {and the length word} correspond to every ball in $\mathcal{G}_w$ are all equal to $N$. Let
$$
{R}(Z, \gamma, \delta, N) :=\frac{1}{m^{N}} \sum_{\lvert w\rvert=N} {R}_{w}(Z, \gamma, \delta, N) ,
$$
and set
$$
\overline{r}(Z, \gamma, \delta):=\limsup _{N \rightarrow \infty} {R}(Z, \gamma, \delta, N).
$$

The critical values $h_Z(G,\delta)$ and $\overline{Ch}_Z(G,\delta)$ are defined as
$$
\begin{aligned}
	h_Z(G,\delta)&:=\inf\{\gamma:{m}(Z,\gamma,\delta)=0\}=\sup \{\gamma:{m}(Z,\gamma,\delta)=\infty\},\\
	\overline{Ch}_{Z}(G,\delta)&:=\inf \{\gamma:\,\, \overline{r}(Z, \gamma, \delta)=0\}=\sup \{\gamma: \,\,\overline{r}(Z, \gamma, \delta)=\infty\}.
\end{aligned}
$$ 
The topological entropy and upper capacity topological entropy of $Z$ of free semigroup action $G$ are then defined as 
$$
\begin{aligned}
	h_Z(G)&:=\lim _{\delta \rightarrow 0}h_{Z}(G,\delta),\\
	\overline{Ch}_{Z}(G)&:=\lim _{\delta \rightarrow 0} \overline{Ch}_{Z}(G,\delta).
\end{aligned}
$$

\begin{remark}
	Let $f: X \to X$ be a continuous transformation and $G$ the free semigroup generated by the map $f$. Then $h_{Z}(G)=h_{Z}(f)$,  $\overline{C h}_{Z}(G)=\overline{C h}_{Z}(f)$, for any set $Z \subset X$, where $h_{Z}(f)$ and $\overline{C h}_{Z}(f)$ are the topological entropy and upper capacity topological entropy defined by Pesin \cite{MR1489237}. If $Z=X$,  then $h(G)=h(f)=h_X(f)=\overline{Ch}_X(f)$, i.e., the classical topological entropy defined by Adler et al \cite{MR175106}.
\end{remark}

In \cite{MR3918203}, 
{they proved that the upper} capacity topological entropy of the skew product $F$ satisfies the following result for any subset $Z\subseteq X$.
\begin{theorem}(\cite{MR3918203}, Theorem 5.1)
	\label{Topological upper capacity entropy skew product}
	For any subset $Z \subset X$, then
	$$
	\overline{Ch}_{\Sigma_{m}^+ \times Z}(F)=\log m+\overline{Ch}_{Z}(G) .
	$$
\end{theorem}
\begin{remark}
	If $Z=X,$ the authors of  \cite{MR3918203} proved that $h(G)=\overline{Ch}_X(G)$. Hence, we have
	$$
	\overline{Ch}_{\Sigma_{m}^+ \times X}(F)=\log m+\overline{Ch}_{X}(G) .
	$$
	Since $h(F)=h_{\Sigma_{m}^+ \times X}(F)=\overline{Ch}_{\Sigma_{m}^+ \times X}(F)$, then Theorem \ref{Topological entropy skew product} can be restated as
	$$
	h_{\Sigma_{m}^+ \times X}(F)=\log m+h(G) .
	$$
\end{remark}

\subsection{Stationary measure}\label{Stationary measure}
Let $\mathbf{p}:=(p_0,\cdots,p_{m-1})$ be a probability vector with non-zero entries (i.e., $p_j>0$ for each $j$ and $\sum_{j=0}^{m-1}p_j=1$). The Bernoulli measure $\mathbb{P}$ on $\Sigma_{m}^+$ generated {by the probability vector}
$\mathbf{p}$ is $\sigma$-invariant and ergodic. Given a point $x\in X$ and measurable set $A\subseteq X$, the transition probabilities {are defined by the formula}
$$
\mathcal{P}(x, A)=\int \chi_{A}\left(f_{j}(x)\right) \mathbf{p}(\mathrm{d} j),
$$
where $\chi_{A}$ denotes the indicator map corresponding with the set $A$. Let $\mathcal{M}(X)$ denote the set of all probability measures on $X$. For every probability measure $\mu\in \mathcal{M}(X)$, the adjoint operator $\mathcal{P}^{*}$ is defined by the following way,
$$
\begin{aligned}
	\mathcal{P}^{*} \mu(A)=\int \mathcal{P}\left(x, {A}\right) d \mu(x)=\int \int \chi_{A}\left(f_{j}(x)\right)\mathbf{p}(\mathrm{d} j) \mu(\mathrm{d}x)
	=\sum_{j=0}^{m-1}p_j\mu\left (f_j^{-1}A\right ).
\end{aligned}
$$
A Borel probability measure $\mu\in\mathcal{M}(X)$ is said to be $\mathbb{P}$-stationary if, $\mathcal{P}^{*} \mu=\mu$.

As $X$ is a compact metric space, the set of $\mathbb{P}$-stationary probability measures is a nonempty  compact convex set with respect to the weak$^{*}$ topology for every $\mathbb{P}$. Its extreme points are called $\mathbb{P}$-ergodic. For more information, see \cite{MR884892}. When convenient, we will use the following criterium :
\begin{proposition} \label{invariant measure of random}
	(\cite{MR884892}, Lemma I.2.3)
	Let $\mathbb{P}$ be a Bernoulli measure on $\Sigma_{m}^+$, and $\mu$ be a probability measure on $X$, then 
	\begin{itemize}
		\item [(1)] \label{invariant} $\mu$ is $\mathbb{P}$-stationary if and only if the product probability measure $\mathbb{P} \times \mu$ is $F$-invariant.
		\item [(2)] $\mu$ is $\mathbb{P}$-stationary and ergodic if and only if the product probability measure $\mathbb{P} \times \mu$ is $F$-invariant and ergodic.
	\end{itemize}
\end{proposition}

\section{Periodic-like recurrence and $\mathbf{g}$-almost product property of free semigroup actions}\label{3}
In this section, we introduce the new concept of $\mathbf{g}$-almost product property of free semigroup actions, and some concepts of transitive points, quasiregular points, upper recurrent points and Banach upper recurrent points with respect to a certain orbit of free semigroup actions.
We obtain that the $\mathbf{g}$-almost product property is weaker than the specification property under free semigroup actions.  The results in this section are inspired by \cite{MR2322186}. Throughout this section we assume that $X$ is a compact metric space, $G$ is the free semigroup generated by $m$ generators $f_0,\cdots,f_{m-1}$ which are continuous maps on $X$ and $F$ is the skew product map corresponding to the maps $f_0,\cdots,f_{m-1}$.

Let us introduce the definitions of recurrence for free semigroup actions.
\begin{definition}
	Given $\iota=(i_0,i_1,\cdots)\in\Sigma_{m}^+$, a point $x\in X$ is called {a transitive point}
	 with respect to $\iota$ of free semigroup action $G$ if  the orbit of $x$ under $\iota$,
	$$
	orb(x,\iota,G):=\{x, f_{i_0}(x),f_{i_1i_0}(x),\cdots\}
	$$
	is dense in $X$.
\end{definition}

\begin{definition}
	Given $\iota=(i_0,i_1,\cdots)\in\Sigma_{m}^+$, a point $x\in X$ is called {a quasiregular point}
	  with respect to $\iota$ of free semigroup action $G$ if a sequence
	$$
	\frac{1}{n}\sum_{j=0}^{n-1}\delta_{F^j(\iota,x)}
	$$
	{converges} in the weak$^{*}$ topology. 
\end{definition}
Denote by $\mathrm{Tran}(\iota,G)$ and $\mathrm{QR}(\iota,G)$ the sets of the transitive points and the quasiregular points with respect to $\iota$ 
{of free semigroup action,}
respectively. We write $\mathrm{Tran}(G)$ and $\mathrm{QR}(G)$ for the union of $\mathrm{Tran}(\iota,G)$ and $\mathrm{QR}(\iota,G)$ for all $\iota$, respectively.

Let $U\subset X$ be a nonempty open set and $x\in X$, $\iota=(i_0,i_1,\cdots) \in\Sigma_{m}^+$, the set of visiting time with respect to $\iota$ is defined by
$$
N_\iota (x,U):=\left \{n\in\mathbb{N}: f_{i_{n-1}\cdots i_0}(x)\in U \right \}.
$$
\begin{definition}
	Given $\iota=(i_0,i_1,\cdots)\in\Sigma_{m}^+$,  a point $x\in X$ is called {a upper recurrent point}
	 with respect to $\iota$ of free semigroup action $G$ if for any $\varepsilon>0$, the set of visiting time $N_\iota\left (x,B(x,\varepsilon)\right )$ has a positive upper density.
\end{definition}

\begin{definition}
	Given $\iota=(i_0,i_1,\cdots)\in\Sigma_{m}^+$,  a point $x\in X$ is called {a Banach upper} recurrent point
	 with respect to $\iota$ of free semigroup action $G$ if for any $\varepsilon>0$, the set of visiting time $N_\iota\left (x,B(x,\varepsilon)\right )$ has a positive Banach upper density. 
\end{definition}

Denote by $\mathrm{QW}(\iota,G)$ and $\mathrm{BR}(\iota,G)$ the sets of the upper recurrent points and the Banach upper recurrent points with respect to $\iota$ of free semigroup action $G$, respectively. Let 
$$
\mathrm{QW}(G):=\bigcup_{\iota\in\Sigma_{m}^+}\mathrm{QW}(\iota,G), \,\mathrm{BR}(G):=\bigcup_{\iota\in\Sigma_{m}^+}\mathrm{BR}(\iota,G).
$$
Let us call $\mathrm{QW}(G)$ and $\mathrm{BR}(G)$ the sets of the upper recurrent points and the Banach upper recurrent points of free semigroup action, respectively. It is easy to check that $\mathrm{QW}(G)$ coincides with the set of  the quasi-weakly almost periodic points of free semigroup action defined by Zhu and Ma \cite{MR4200965}. Clearly,
$$
\mathrm{QW}(\iota,G) \subseteq \mathrm{BR}(\iota,G).
$$

The notion of specification, introduced by Bowen \cite{MR282372}, says that one can always find a single orbit to interpolate between different pieces of orbits.  In the case of $\beta$-shifts it is known that the specification property holds for a set of $\beta$ of Lebesgue measure zero (see \cite{MR1452189}). In \cite{MR2322186}, the authors studied a new condition, called $\mathbf{g}$-almost product product property, which is weaker that specification property,  and proved the $\mathbf{g}$-almost product product  property always {holds for $\beta$-shifts. }

{Next we introduce} the concept of {$\mathbf{g}$}-almost product property of free semigroup actions:
\begin{definition}
	Let $\mathbf{g}: \mathbb{N} \rightarrow \mathbb{N}$ be a given nondecreasing unbounded map with the properties
	$$
	\mathbf{g}(n)<n \quad \text { and } \quad \lim _{n \rightarrow \infty} \frac{\mathbf{g}(n)}{n}=0 .
	$$
	The function $\mathbf{g}$ is called blowup function.
\end{definition}

Fixed $\varepsilon>0$, $w \in F_{m}^{+}$ and                                                                                                                                       $x \in X$,  define the $\mathbf{g}$-blowup of $B_{w}(x, \varepsilon)$ {as the closed set }
$$
B_{w}(\mathbf{g} ; x, \varepsilon):=\Big \{y\in X : \sharp\left \{ w'\le\overline{w}:d\left (f_{w'}(x),f_{w'}(y)\right )>\varepsilon\right \}<\mathbf{g}\left (\lvert w\rvert+1\right )\Big\}.
$$

\begin{definition}
	We say $G$ satisfies the $\mathbf{g}$-almost product property with the blowup function $\mathbf{g}$, if there exists a nonincreasing function $\mathfrak{m}:\mathbb{R}^+\to\mathbb{N}$, such that for $k\ge 2$, any $k$ points $x_1,\cdots,x_k\in X$, any positive $\varepsilon_1,\cdots,\varepsilon_k$, and any words $w_{(\varepsilon_1)},\cdots,w_{(\varepsilon_{k})}\in F_m^+$ with $\lvert w_{(\varepsilon_1)}\rvert \ge \mathfrak{m}(\varepsilon_{1}),\cdots,\lvert w_{(\varepsilon_{k})}\rvert \ge \mathfrak{m}(\varepsilon_{k})$,
	$$
	B_{w_{(\varepsilon_{1})}}(\mathbf{g};x_1,\varepsilon_{1})\cap\left (\bigcap_{j=2}^{k} f^{-1}_{\overline{w_{(\varepsilon_{j-1})}}\cdots \overline{w_{(\varepsilon_1)}}} B_{w_{(\varepsilon_{j})}}\left(\mathbf{g} ; x_{j}, \varepsilon_{j}\right)\right ) \neq \emptyset.
	$$
\end{definition}

{Under $\mathbf{g}$-almost product property, the topological entropy of periodic-like recurrent sets has been studied in  \cite{MR3963890}, but the topological entropy of such sets has not been studied in dynamical systems of  free semigroup actions. In this paper, we focus on the topological entropy of similar sets of free semigroup actions and obtain more extensive results. Therefore, it is important and necessary to introduce the $\mathbf{g}$-almost product property of free semigroup actions.}

If $m=1$,  the $\mathbf{g}$-almost product property of free semigroup actions coincides with the definition introduced by Pfister and Sullivan \cite{MR2322186,MR2109476}.

The next proposition asserts the relationship between specification property and $\mathbf{g}$-almost product property of free semigroup actions.
\begin{proposition}\label{proposition almost}
	Let $\mathbf{g}$ be any blowup function {and $G$ satisfies the specification} property. Then it has the $\mathbf{g}$-almost product property.
\end{proposition}
\begin{proof}
	This proof extends the method of Proposition 2.1 in \cite{MR2322186} to the free semigroup actions, but we provide the complete proof for the reader’s convenience. 

	Let $\mathfrak{p}(\varepsilon)$ be the positive integer in the definition of specification property of $G$ (see Sec. \ref{the concepts of free semigroup actions}) for $\varepsilon>0$. It is no restriction to suppose that the function $\mathfrak{p}(\varepsilon)$ is nonincreasing. Let $\left\{x_{1}, \cdots, x_{k}\right\}$ and $\left\{\varepsilon_{1}, \cdots, \varepsilon_{k}\right\}$ be given. Let $\delta_{r}:=2^{-r}, r \in \mathbb{N}$. Next, we may define 
	{a nonincreasing function $\mathfrak{m} :\mathbb{R}^+\to\mathbb{N}$}
	as follows: $$\mathfrak{m}(\varepsilon):=\widetilde{\mathfrak{m}}(2\delta_r),$$
	where $r=\min\{i:2 \delta_{i} \leq \varepsilon \}$ and  $\widetilde{\mathfrak{m}}(2\delta_r):= \min\left \{m: \mathbf{g}(m)\ge 2 \mathfrak{p}(\delta_r)\right \}$.
	
	It is sufficient to prove the statement for $\varepsilon_{j}$ of the form $2 \delta_{r_j}, j=1, \cdots, k$, where, as above $r_j=\min\{i: 2\delta_i\le \varepsilon_{j}\}$. Precisely, if $\varepsilon_{j}$ is not of that form, we change it into $2 \delta_{r_j}$. From now on we assume that, for all $j$, $\varepsilon_{j}$ is of the form $2 \delta_{r_j}$. Let $w_{(\varepsilon_1)},\cdots,w_{(\varepsilon_{k})}$ be the words with the length not less than $\mathfrak{m}(\varepsilon_{1}),\cdots,\mathfrak{m}(\varepsilon_{k})$, respectively. Let $n_1,\cdots,n_k$ denote the length of $w_{(\varepsilon_1)},\cdots,w_{(\varepsilon_{k})}$, respectively.
	
	We prove the proposition by an iterative construction. Let $\Delta\left(x_{j}\right):=\varepsilon_{j} / 2=\delta_{r_j}$, $w(x_j):=w_{(\varepsilon_j)}$, $p(x_j):=\mathfrak{p}\left(\Delta(x_j)\right)$ and $n(x_{j}):=n_{j}$. The sequence $\{x_{1}, \cdots, x_{k}\}$ is considered as an ordered sequence; its elements are called original points. The possible values of $\Delta\left(x_{j}\right)$ are rewritten $\Delta_{1}>\Delta_{2}>\cdots>\Delta_{q}$. A level-$i$ point is defined by an original point $x_{j}$ such that $\Delta\left(x_{j}\right)=\Delta_{i} .$

	At step 1 we consider the level-1 points labeled by
	$$
	S_{1}:=\left\{j \in[1, k]: \Delta\left(x_{j}\right)=\Delta_{1}\right\} .
	$$
	If $S_{1}=[1, k]$, then by the specification property there exists $y$ such that
	{$$d\left (f_{\overline{w{(x_1)}\lvert_{[0,i]}}}(x_1),f_{\overline{w{(x_1)}\lvert_{[0,i]}}}(y)\right )\le\Delta_{1},\quad i=p(x_1),\cdots, n(x_1)-p(x_1)-1,$$
	$$d\big(f_{\overline{w{(x_j)}\lvert_{[0,i]}}}(x_j) ,f_{\overline{w{(x_j)}\lvert_{[0,i]}}\,\overline{w{(x_{j-1})}}\cdots\overline{w{(x_1)}}}(y)\big )\le\Delta_{1},i=p(x_j),\cdots, n(x_j)-p(x_j)-1,
$$}where $j=2,\cdots,k$, which proves this case by the definition of the function $\mathfrak{m}$. If $S_{1} \neq[1, k]$, then we decompose it into maximal subsets of consecutive points, called components. (The components are defined with respect to the whole sequence.) Let $J$ be a component, say $[r, s]$ with $r<s$. By the specification property there exists $y$ such that
	{$$d\left (f_{\overline{w{(x_r)}\lvert_{[0,i]}}}(x_r),f_{\overline{w{(x_r)}\lvert_{[0,i]}}}(y)\right )\le\Delta_{1},\quad i=p(x_r),\cdots, n(x_r)-p(x_r)-1,$$
	$$
d\big (f_{\overline{w{(x_j)}\lvert_{[0,i]}}}(x_j) ,f_{\overline{w{(x_j)}\lvert_{[0,i]}}\,\overline{w{(x_{j-1})}}\cdots\overline{w{(x_r)}}}(y)\big)\le\Delta_{1}, i=p(x_j),\cdots, n(x_j)-p(x_j)-1,
$$}where $j=r+1,\cdots,s$. Hence,
	$$
	y\in 
	B_{w{(x_r)}}(\mathbf{g};x_r,\varepsilon_r)\cap\left (\bigcap_{j=r+1}^s f^{-1}_{\overline{w{(x_{j-1})}}\cdots\overline{w{(x_r)}}}B_{w{(x_j)}}\left (\mathbf{g};x_j,\varepsilon_j\right )\right ).
	$$
	We replace the sequence $\left\{x_{1}, \cdots, x_{k}\right\}$ by the (ordered) sequence
	$$
	\left\{x_{1}, \cdots, x_{r-1}, y, x_{s+1} \cdots, x_{k}\right\}
	$$
	and set, for the concatenated point $y$, let
	$$
	\Delta(y):=\Delta_{1},\, p(y):= \mathfrak{p}(\Delta(y)), \,n(y):=n(x_{r})+\cdots+n(x_{s}),\,w(y):=w(x_r)\cdots w(x_s).
	$$
	We do this operation for all components which are not singletons. After these operations we have a new (ordered) sequence $\left\{z_{1}, \cdots, z_{k_{1}}\right\}, k_{1} \leq k$, where the point $z_{i}$ is either a point of the original sequence, or a concatenated point. This ends the construction at step 1.

	Let
	$$
	S_{2}:=\left\{j \in\left[1, k_{1}\right]: \Delta\left(z_{j}\right) \geq \Delta_{2}\right\} .
	$$
	We decompose this set into components. Let $[r, s]$ be a component which is not a singleton $(r<s)$. We replace that component by a single concatenated point $y$ such that if $z_r$ is concatenated point of  $S_1$,
	$$
	d\left (f_{\overline{w{(z_r)}\lvert_{[0,i]}}}(z_r),f_{\overline{w{(z_r)}\lvert_{[0,i]}}}(y)\right )\le\Delta_{2},\quad i=0,\cdots, n(z_r)-1,
	$$
	otherwise,
	$$
	d\left (f_{\overline{w{(z_r)}\lvert_{[0,i]}}}(z_r),f_{\overline{w{(z_r)}\lvert_{[0,i]}}}(y)\right )\le\Delta_{2},\quad i=p(z_r),\cdots, n(z_r)-p(z_r)-1;
	$$
	and, for $j=r+1,\cdots,s$, if $z_j$ is concatenated point of  $S_1$,
	$$
	d\left (f_{\overline{w{(z_j)}\lvert_{[0,i]}}}(z_j),f_{\overline{w{(z_j)}\lvert_{[0,i]}}\,\overline{w(z_{j-1})}\cdots\overline{w(z_r)}}(y)\right )\le\Delta_{2},\quad i=0,\cdots, n(z_j)-1,
	$$
	otherwise,
	{	$$
		d\big (f_{\overline{w{(z_j)}\lvert_{[0,i]}}}(z_j),f_{\overline{w{(z_j)}\lvert_{[0,i]}}\,\overline{w(z_{j-1})}\cdots\overline{w(z_r)}}(y)\big )\le\Delta_{2}, i=p(z_j),\cdots, n(z_j)-p(z_j)-1.
		$$}Existence of such a $y$ is a consequence of the specification property. We set
	$$
	\Delta(y):=\Delta_{2},\, p(y):= \mathfrak{p}(\Delta(y)),\, n(y):=n(z_{r})+\cdots+n(z_{s}), \,w(y):=w(z_r)\cdots w(z_s).
	$$
	The construction of $y$ involves consecutive points of the original sequence (via the concatenated points), say points $x_{j}, j \in[u, t] .$ Since $\delta_{j}=\sum_{i>j} \delta_{i}$,
	$$
	y\in 
	B_{w{(x_u)}}(\mathbf{g};x_u,\varepsilon_u)\cap\left (\bigcap_{j=u+1}^t f^{-1}_{\overline{w{(x_{j-1})}}\cdots\overline{w{(x_u)}}}B_{w{(x_j)}}\left (\mathbf{g};x_j,\varepsilon_j\right )\right ).
	$$
	We do these operations for all components of $S_{2}$, which are not singletons. We get a new ordered sequence, still denoted by $\left\{z_{1}, \cdots, z_{k_{2}}\right\}$. This ends the construction at level 2.

	The construction at level 3 is similar to the construction at level 2, using
	$$
	S_{3}:=\left\{j\in\left[1, k_{2}\right]: \Delta\left(z_{j}\right) \geq \Delta_{3}\right\} .
	$$
	Once step $q$ is completed,  we have a single concatenated point $y$ such that
	{	
		$$
		d\left (f_{\overline{w{(x_1)}\lvert_{[0,i]}}}(x_1),f_{\overline{w{(x_1)}\lvert_{[0,i]}}}(y)\right )\le\varepsilon_{1},\quad i=p(x_1),\cdots, n(x_1)-p(x_1)-1,
		$$
		$$
		d\big (f_{\overline{w{(x_j)}\lvert_{[0,i]}}}(x_j) ,f_{\overline{w{(x_j)}\lvert_{[0,i]}}\,\overline{w{(x_{j-1})}}\cdots\overline{w{(x_1)}}}(y)\big )\le\varepsilon_j, i=p(x_j),\cdots, n(x_j)-p(x_j)-1,
		$$}where $j=2,\cdots,k.$ Observe that, for all $j$, $\mathbf{g}(n_j)\ge 2p(x_j)$. As a consequence, 
	\begin{equation}\label{Prop-1}
			y\in B_{w_{(\varepsilon_{1})}}(\mathbf{g};x_1,\varepsilon_{1})\cap\left (\bigcap_{j=2}^{k} f^{-1}_{\overline{w_{(\varepsilon_{j-1})}}\cdots \overline{w_{(\varepsilon_1)}}} B_{w_{(\varepsilon_{j})}}\left(\mathbf{g} ; x_{j}, \varepsilon_{j}\right)\right ).
	\end{equation}
\end{proof}

\begin{remark}\label{saturated}
	If $m=1$, it generates the Proposition 2.1 in \cite{MR2322186}.
\end{remark}

In \cite{MR3503951},  Rodrigues and Varandas proved that if $X$ is a compact Riemannian manifold, and $G$ is free semigroup generated by ${f_0,\cdots,f_{m-1}}$ which are all expanding maps, then $G$ satisfies the specification property, furthermore, it has the $\mathbf{g}$-almost product property by Proposition \ref{proposition almost}.  

Next, we describe an example to help us interpret the $\mathbf{g}$-almost product property of free semigroup actions.  
{
\begin{example}\label{example-1}
	Let $M$ be a compact Riemannian manifold and $G$ the free semigroup generated by $f_0,\cdots,f_{m-1}$ on $M$ which are $C^1$-local diffeomorphisms such that for any $j=0,\cdots,m-1$, $\|Df_j(x)v\|\ge\lambda_j\|v\|$ for all $x\in M$ and all $v\in T_x M$, where $\lambda_j$ is a constant larger than 1.  It follows from \cite{MR4200965} and Theorem 16 of \cite{MR3503951} that $G$ satisfies positively expansive and specification property. Let $\mathbf{g}$ be a blowup function. Consider the nonincreasing function $\mathbf{m}:\mathbb{R}^+\to\mathbb{N}$ given by Proposition \ref{proposition almost}. For $k\ge 2$, let $x_1,\cdots,x_k\in X$,  $\varepsilon_1,\cdots,\varepsilon_k>0$, and  $w_{(\varepsilon_1)},\cdots,w_{(\varepsilon_{k})}\in F_m^+$ with $\lvert w_{(\varepsilon_1)}\rvert \ge \mathfrak{m}(\varepsilon_{1}),\cdots,\lvert w_{(\varepsilon_{k})}\rvert \ge \mathfrak{m}(\varepsilon_{k})$ be given. By the formula (\ref{Prop-1}) in Proposition \ref{proposition almost}, we have that 
		$$
		B_{w_{(\varepsilon_{1})}}(\mathbf{g};x_1,\varepsilon_{1})\cap\left (\bigcap_{j=2}^{k} f^{-1}_{\overline{w_{(\varepsilon_{j-1})}}\cdots \overline{w_{(\varepsilon_1)}}} B_{w_{(\varepsilon_{j})}}\left(\mathbf{g} ; x_{j}, \varepsilon_{j}\right)\right )\neq\emptyset.
		$$
		Hence $G$ satisfies the $\mathbf{g}$-almost product property for any blowup function $\mathbf{g}$. 
\end{example}}

\begin{proposition}
	\label{2g-almost product property}
	If $G$ satisfies the $\mathbf{g}$-almost product property, then the skew product map $F$ corresponding to the maps $f_0,\cdots,f_{m-1}$ has $2\mathbf{g}$-almost product property.
\end{proposition}

\begin{proof}
	The shift map $\sigma:\Sigma_m^+\to\Sigma_m^+$ has specification property (see \cite{MR0457675}).  Let $\mathfrak{p}(\varepsilon)$ be the positive integer in the definition of specification property of $\sigma$ for $\varepsilon>0$.  Let $\mathfrak{m}_G$ denote the function in  the $\mathbf{g}$-almost product property for $G$. Let $\delta_{r}:=2^{-r}, \,r\in \mathbb{N}$. It is no restriction to suppose that $\mathfrak{p}(\delta_r)>r$ and the function $\mathfrak{p}(\delta_r)$ is increasing as $r$ increases. Next, we may define a nonincreasing function $\mathfrak{m}_F :\mathbb{R}^+\to\mathbb{N}$ as follows:
	$$\mathfrak{m}_F(\varepsilon):=\widetilde{\mathfrak{m}}(2\delta_r)$$
	where $r=\min\{i :2\delta_i\le\varepsilon\}$ and $\widetilde{\mathfrak{m}}(2\delta_r):= \min \left\{n: \mathbf{g}(n) \geq 2 \mathfrak{p}\left(\delta_{r}\right)\,\text{and } n\ge \mathfrak{m}_G\left(\delta_{r}\right)\right\}$.

	For $k\ge 2$, let $(\iota^{(1)},x_1),\cdots, (\iota^{(k)},x_k)\in\Sigma_m^+\times X$ and $\varepsilon_{1},\cdots,\varepsilon_{k}>0$ be given. It is sufficient to prove the statement for $\varepsilon_{j}$ of the form $2 \delta_{r_j}, j=1, \cdots, k$, where, as above $r_j=\min\{i :2 \delta_{i} \leq \varepsilon_{j}\}$. Precisely, if $\varepsilon_{j}$ is not of that form, we change it into $2 \delta_{r_j}$. From now on we assume that, for all $j, \varepsilon_{j}$ is of the form $2 \delta_{r_j}$. For convenience, write $p_j:=\mathfrak{p}(\delta_{r_j})$ for all $j=1,\cdots,k$.

	For any $n_1\ge \mathfrak{m}_F(\varepsilon_{1}),\cdots, n_k\ge \mathfrak{m}_F(\varepsilon_{k})$,  let $\iota\in \Sigma_{m}^+$ satisfy the following condition:
	$$
	\iota\lvert_{[n_1+n_2+\cdots +n_{j-1},\, n_1+n_2+\cdots +n_j-1]}=\iota^{(j)}\lvert_{[0,\, n_j-1]}, \quad j=1,\cdots, k,
	$$
	where $n_0=0$. We now apply the argument $p_j>r_j$ for all $j$ to obtain
	\begin{equation}\label{3.1}
		d^\prime\left(\sigma^{n_{1}+n_2+\cdots+n_{j-1}+r} (\iota), \sigma^{r} (\iota^{(j)})\right) \leq \varepsilon_{j}, \quad  r=p_{j},p_j+1, \cdots, n_{j}-p_{j}-1.
	\end{equation} 

	Let 
	$$
	\begin{aligned}
		w_{(\varepsilon_{1})}:&=\iota\lvert_{[0, \,n_1-1]}, \\
		w_{(\varepsilon_2)}:&=\iota\lvert_{[n_1,\, n_1+n_2-1]}, \\
		&\vdots  \\
		w_{(\varepsilon_{k})}:&=\iota\lvert_{[n_1+\cdots+n_{k-1},\, n_1+\cdots +n_k-1]}.
	\end{aligned}
	$$
	Observe that $\lvert  w_{(\varepsilon_j)}\rvert \ge \mathfrak{m}_G(\varepsilon_{j})$ for each $j=1,\cdots, k$. The $\mathbf{g}$-almost product property of $G$ implies that
	$$
	B_{w_{(\varepsilon_{1})}}(\mathbf{g};x_1,\varepsilon_{1})\cap\left (\bigcap_{j=2}^{k} f^{-1}_{\overline{w_{(\varepsilon_{j-1})}}\cdots \overline{w_{(\varepsilon_1)}}} B_{w_{(\varepsilon_{j})}}\left(\mathbf{g} ; x_{j}, \varepsilon_{j}\right)\right ) \neq \emptyset.
	$$
	Take an element $x$ from the left set.  For $j=1,\cdots, k$,  define
	{$$
		\Gamma_j:=\bigg \{p_j\le r< n_j-p_j:
		d\big(f_{\overline{w_{(\varepsilon_j)}\lvert_{[0,r]}}\overline{w_{(\varepsilon_{j-1})}}\cdots\overline{w_{(\varepsilon_{1})}}}(x),
		f_{\overline{w_{(\varepsilon_j)}\lvert_{[0,r]}}}(x_j)
		\big ) \le\varepsilon_j \bigg\}.
		$$}To be more precise, for any $r\in\Gamma_j$, 
	\begin{equation}\label{3.2}
		d\left (f_{\overline{\iota\lvert_{[0,\, n_1+n_2+\cdots +n_{j-1}+r]}}}(x),
		f_{\overline{\iota^{(j)}\lvert_{[0,r]}}}(x_j)
		\right )\le\varepsilon_{j}.
	\end{equation}
	Accordingly, 
	{$$
		\begin{aligned}
			&D\Big (F^{n_1+n_2+\cdots +n_{j-1}+r} (\iota,x ) , F^r (\iota^{(j)}, x_j )\Big )\\
			=&D\Big (\big (\sigma^{n_1+n_2+\cdots +n_{j-1}+r}(\iota),\, f_{\overline{\iota\lvert_{[0,\, M_{j-1}+r]}}}(x)  \big), \big(\sigma^r(\iota^{(j)}), f_{\overline{\iota^{(j)}\lvert_{[0,r]}}}(x_j) \big )\Big)\\
			=&\max\Big \{d^\prime\big ( \sigma^{n_1+n_2+\cdots +n_{j-1}+r}(\iota),\sigma^r(\iota_j) \big ),
			d\big( f_{\overline{\iota\lvert_{[0,\, n_1+n_2+\cdots +n_{j-1}+r]}}}(x),f_{\overline{\iota^{(j)}\lvert_{[0,r]}}}(x_j)  \big )\Big \}\\
			\le& \varepsilon_{j}, \quad\text{by these inequations (\ref{3.1}) and (\ref{3.2})}.
		\end{aligned}
		$$}

	Observe that $\sharp\left (\Gamma_j\right )\ge n_j-2p_j-\mathbf{g}(n_j)\ge n_j-2\mathbf{g}(n_j)$. 
	As a consequence, 
	{
	$$
(\iota,x)\in B_{n_1}\big (2\mathbf{g};(\iota^{(1)},x_1 ),\varepsilon_{1}\big) \cap    \bigg (\bigcap_{j=2}^{k} F^{-(n_1+n_2+\cdots +n_{j-1})} B_{n_{j}}\big(2\mathbf{g} ;  (\iota^{(j)}, x_{j}), \varepsilon_{j}\big)\bigg).
$$	
}This proves that $F$ has 2$\mathbf{g}$-almost product property.
\end{proof}

\section{General (ir)regularity}\label{Irregular and regular set}
In this section, we study the more general irregular and regular sets of free semigroup actions and calculate the upper capacity topological entropy of the irregular and regular sets of free semigroup actions. The results in this section are inspired by \cite{MR3963890, MR4200965}. Throughout this section we assume that $X$ is a compact metric space, $G$ is the free semigroup generated by $m$ generators $f_0,\cdots,f_{m-1}$ which are continuous maps on $X$ and $F$ is the skew product map corresponding to the maps $f_0,\cdots,f_{m-1}$.

Let 
$$
R_\alpha(G):=\bigcup_{\iota\in\Sigma_{m}^+} R_\alpha(\iota, G),\quad I_\alpha (G):=\bigcup_{\iota\in\Sigma_{m}^+}I_\alpha(\iota, G).
$$
Let us call $R_\alpha(G)$ and $I_\alpha (G)$ the $\alpha$-regular set and $\alpha$-irregular set of free semigroup actions, respectively. 
\begin{theorem}\label{full entropy 1-1}
	Let $(X,d)$ be a compact metric space and $G$ the free semigroup action on $X$ generated by $f_0,\cdots,f_{m-1}$. Let $\alpha: \mathcal{M}(\Sigma_{m}^+\times X, F)\to\mathbb{R}$ be a continuous function. Then,
	$$
	\overline{Ch}_{R_\alpha(G)}(G)=\overline{Ch}_X(G)=h(G).
	$$
\end{theorem}

\begin{proof}
	Consider a set
	$$
	R_\alpha(F):=\left \{(\iota,x)\in\Sigma_{m}^+\times X: \inf_{\nu\in M_{(\iota,x)}(F)}\alpha(\nu)= \sup_{\nu\in M_{(\iota,x)}(F)}\alpha(\nu)\right \}.
	$$
	It follows from Theorem 4.1(4) of \cite{MR3963890} that
	\begin{equation}\label{full entropy1-1-1}
		h_{R_\alpha(F)}(F)=h_{\Sigma_{m}^+\times X}(F).
	\end{equation}
	For $(\iota,x)\in R_\alpha(F)$, it is immediate that $x\in R_\alpha(\iota,G)$, then $x\in R_\alpha(G)$.  This implies that
	$$
	R_\alpha(F)\subseteq \Sigma_{m}^+\times R_\alpha(G)\subseteq \Sigma_{m}^+\times X.
	$$
	In this way we conclude from the formula (\ref{full entropy1-1-1}) that
	\begin{equation}\label{full entropy1-1-2}
		h_{\Sigma_{m}^+\times X}(F)=h_{R_\alpha(F)}(F)\le \overline{Ch}_{\Sigma_{m}^+\times R_\alpha(G)}(F).
	\end{equation}
          From Theorem \ref{Topological entropy skew product}, we obtain that 
          \begin{equation}\label{full entropy1-1-4}
          	\log m +h(G)=h_{\Sigma_{m}^+\times X}(F).
          \end{equation}
	By Theorem \ref{Topological upper capacity entropy skew product}, one has
	\begin{equation}\label{full entropy1-1-3}
		\overline{Ch}_{\Sigma_{m}^+\times R_\alpha(G)}(F)= \log m +\overline{Ch}_{R_\alpha(G)}(G).
	\end{equation}
	Combining the equations (\ref{full entropy1-1-2}), (\ref{full entropy1-1-4}) and (\ref{full entropy1-1-3}), we get that
	$$
	\begin{aligned}
		\log m+h(G)&=h_{R_\alpha(F)}(F)\\
		&\le\overline{Ch}_{\Sigma_{m}^+\times R_\alpha(G)}(F)\\
		&=\log m +\overline{Ch}_{R_\alpha(G)}(G).
	\end{aligned}
	$$
	Hence, 
	$$
	\overline{Ch}_X(G)=h(G)\le \overline{Ch}_{R_\alpha(G)}(G).
	$$
	
	Obviously,
	$$
	\overline{Ch}_{R_\alpha(G)}(G)\le \overline{Ch}_X(G)=h(G).
	$$
	Consequently,
	$$
	\overline{Ch}_{R_\alpha(G)}(G)=\overline{Ch}_X(G)=h(G).
	$$
\end{proof}
\begin{theorem}\label{full entropy 2-2}
	Suppose that $G$ has the $\mathbf{g}$-almost product property, there exists a $\mathbb{P}$-stationary measure with full support where $\mathbb{P}$ is a Bernoulli measure on $\Sigma_{m}^+$. Let $\alpha: \mathcal{M}(\Sigma_{m}^+\times X, F)\to\mathbb{R}$ be a continuous function satisfying the condition A.3. If $\inf_{\nu \in \mathcal{M}(\Sigma_{m}^+\times X, F)}\alpha(\nu)<\sup_{\nu \in \mathcal{M}(\Sigma_{m}^+\times X, F)}\alpha (\nu)$, then
	$$
	\overline{Ch}_{I_\alpha(G)}(G)=\overline{Ch}_{E(I_\alpha,\mathrm{Tran})} (G)=\overline{Ch}_X(G)=h(G),
	$$
	where $E(I_\alpha,\mathrm{Tran}):=\cup_{\iota\in\Sigma_{m}^+}\left ( I_\alpha(\iota,G)\cap \mathrm{Tran}(\iota,G)\right )$. In particular,
	$$
	\overline{Ch}_{I_\alpha(G)}(G)=\overline{Ch}_{I_\alpha (G)\cap\mathrm{Tran}(G)}(G)=\overline{Ch}_X(G)=h(G).
	$$
\end{theorem}
\begin{proof}
	Suppose $\mu$ is the $\mathbb{P}$-stationary measure with full support. Then, Proposition \ref{invariant measure of random} ensures that $\mathbb{P}\times \mu$ is an invariant measure under the skew product $F$ with support $\Sigma_{m}^+\times X$. From Proposition \ref{2g-almost product property}, the skew product $F$ has 2$\mathbf{g}$-almost product property. Consider a set 
	$$
	I_\alpha(F):=\left \{(\iota,x)\in\Sigma_{m}^+\times X: \inf_{\nu\in M_{(\iota,x)}(F)}\alpha(\nu)< \sup_{\nu\in M_{(\iota,x)}(F)}\alpha(\nu)\right \}.
	$$
	Hence, from Theorem 4.1 (2) of \cite{MR3963890},  one has 
	\begin{equation}\label{full entropy2-2-1}
		h_{I_\alpha(F)}(F)=h_{I_\alpha (F)\cap \mathrm{Tran}(F)}(F) =h_{\Sigma_{m}^+\times X}(F).
	\end{equation}
	It is clear that if $(\iota,x)\in I_\alpha (F)\cap \mathrm{Tran}(F)$, then $x\in I_\alpha (\iota,G)\cap \mathrm{Tran}(\iota,G)$. Accordingly, $x\in E(I_\alpha,\mathrm{Tran})$. This yields that
	$$
	I_\alpha (F)\cap \mathrm{Tran}(F) \subseteq \Sigma_{m}^+\times E(I_\alpha,\mathrm{Tran}) \subseteq \Sigma_{m}^+\times X.
	$$
	In this way we conclude from the formula (\ref{full entropy2-2-1}) that
	\begin{equation}\label{full entropy2-2-2}
		h_{\Sigma_{m}^+\times X}(F)=h_{I_\alpha (F)\cap \mathrm{Tran}(F)}(F)\le \overline{Ch}_{\Sigma_{m}^+\times E(I_\alpha,\mathrm{Tran})}(F).
	\end{equation}
	By Theorem \ref{Topological upper capacity entropy skew product}, one has
	\begin{equation}\label{full entropy2-2-3}
		\overline{Ch}_{\Sigma_{m}^+\times E(I_\alpha,\mathrm{Tran})}(F)= \log m +\overline{Ch}_{E(I_\alpha,\mathrm{Tran})}(G).
	\end{equation}
	Combining the equations (\ref{full entropy1-1-4}), (\ref{full entropy2-2-2}) and (\ref{full entropy2-2-3}), we get that
	$$
	\begin{aligned}
		\log m+h(G)&=h_{I_\alpha (F)\cap \mathrm{Tran}(F)}(F)\\
		&\le\overline{Ch}_{\Sigma_{m}^+\times E(I_\alpha,\mathrm{Tran})}(F)\\
		&=\log m +\overline{Ch}_{E(I_\alpha,\mathrm{Tran})}(G).
	\end{aligned}
	$$
	Hence, 
	$$
	\overline{Ch}_X(G)=h(G)\le \overline{Ch}_{E(I_\alpha,\mathrm{Tran})}(G).
	$$
	
	Obviously,
	$$
	\overline{Ch}_{E(I_\alpha,\mathrm{Tran})}(G)\le \overline{Ch}_X(G)=h(G).
	$$
	Consequently,
	$$
	\overline{Ch}_{E(I_\alpha,\mathrm{Tran})}(G)=\overline{Ch}_X(G)=h(G).
	$$
	We may obtain from  $E(I_\alpha,\mathrm{Tran})\subseteq I_\alpha(G)$ that 
	$$
	\overline{Ch}_{I_\alpha(G)}(G)=\overline{Ch}_{E(I_\alpha,\mathrm{Tran})}(G)=\overline{Ch}_X(G)=h(G).
	$$
	
	It is easy to check that
	$$
	\begin{aligned}
		I_\alpha(G)\cap \mathrm{Tran}(G)&=\left ( \bigcup_{\iota\in\Sigma_{m}^+} I_\alpha(\iota,G)\right )\cap \left ( \bigcup_{\iota^\prime\in\Sigma_{m}^+} \mathrm{Tran}(\iota^\prime,G)\right )\\
		&= \bigcup_{\iota,\iota^\prime\in\Sigma_{m}^+}\left ( I_\alpha(\iota,G)\cap \mathrm{Tran}(\iota^\prime,G)\right )\\
		&\supseteq \bigcup_{\iota\in\Sigma_{m}^+} \left (I_\alpha(\iota,G)\cap \mathrm{Tran}(\iota,G)\right )=E(I_\alpha,\mathrm{Tran}).
	\end{aligned}
	$$
	Hence, 
	$$
	\overline{Ch}_{I_\alpha(G)}(G)=\overline{Ch}_{I_\alpha(G)\cap \mathrm{Tran}(G)}(G)=\overline{Ch}_X(G)=h(G).
	$$
\end{proof}

\begin{remark}
	Both Theorem \ref{full entropy 1-1} and \ref{full entropy 2-2} are extension of Theorem 4.1 of \cite{MR3963890}.
\end{remark}

For $\iota=(i_0,i_1,\cdots)\in \Sigma^+_m$, consider a set
$$
I_\varphi(\iota,G):=\left \{ x\in X: \lim_{n\to \infty}\frac{1}{n}\sum_{j=0}^{n-1}\varphi\left (f_{i_{j-1}\cdots i_0}(x)\right )\text{ does not exist}\right \}.
$$
Let $R_\varphi(\iota,G):= X \backslash I_\varphi(\iota,G)$, and 
$$
R_\varphi(G):=\bigcup_{\iota\in\Sigma_{m}^+} R_\varphi(\iota,G),\quad I_\varphi(G):=\bigcup_{\iota\in\Sigma_{m}^+} I_\varphi(\iota,G).
$$
{It is easy to find that $I_\varphi(G)$ coincides with the $\varphi$-irregular set of free semigroup action defined by Zhu and Ma \cite{MR4200965}.}
For convenience, we call $R_\varphi(G)$ to be $\varphi$-regular set of free semigroup action.

For a continuous function $\varphi:X\to\mathbb{R}$, consider a function $\psi:\Sigma_{m}^+\times X\to\mathbb{R}$ such that for any $\iota=(i_0,i_1,\cdots)\in\Sigma_{m}^+$, the map $\psi$ satisfies $\psi(\iota,x)=\varphi(x)$, then $\psi$ is continuous. The continuous function $\alpha :\mathcal{M}(\Sigma_{m}^+\times X, F)\to\mathbb{R}$ is given by $\alpha(\nu)=\int \psi \mathrm{d}\nu$. It is easy to check that the function $\alpha$ satisfies the conditions A.1, A.2 and A.3. It follows from the definition of the function $\psi$ that the limit
$$
\lim_{n\to \infty}\frac{1}{n}\sum_{j=0}^{n-1}\varphi\left (f_{i_{j-1}\cdots i_0}(x)\right )
$$
exists if and only if  
$$
\inf_{\nu\in M_{(\iota,x)}(F)}\alpha(\nu)= \sup_{\nu\in M_{(\iota,x)}(F)}\alpha(\nu).
$$
Hence, $R_{\alpha}(\iota,G)=R_{\varphi}(\iota,G)$. Analogously, $I_\alpha (\iota,G)=I_{\varphi}(\iota,G)$.

\begin{corollary}\label{full entropy 1}
	Let $(X,d)$ be a compact metric space and $G$ the free semigroup action on $X$ generated by $f_0,\cdots,f_{m-1}$. Then the $\varphi$-regular set of free semigroup {action}
	carries full upper capacity topological entropy, that is,
	$$
	\overline{Ch}_{R_\varphi(G)}(G)=\overline{Ch}_X(G)=h(G).
	$$
\end{corollary}

If $m = 1$, {then the above corollary} coincides with  {the result of Theorem 4.2} that Tian proved in \cite{MR3436391}.

\begin{corollary}\label{full entropy 2}
	Suppose that $G$ has the $\mathbf{g}$-almost product property, there exists a $\mathbb{P}$-stationary measure with full support where $\mathbb{P}$ is a Bernoulli measure on $\Sigma_{m}^+$. Let $\varphi :X\to\mathbb{R}$ be a continuous function. If {$I_\varphi(\iota,G)$} is non-empty for some $\iota\in\Sigma_{m}^+$, then
	$$
	\overline{Ch}_{I_\varphi(G)}(G)=\overline{Ch}_{E(I_\varphi,\mathrm{Tran})}(G)=\overline{Ch}_X(G)=h(G),
	$$
	where ${E(I_\varphi,\mathrm{Tran})}:=\cup_{\iota\in\Sigma_{m}^+}\left ( I_\varphi(\iota,G)\cap \mathrm{Tran}(\iota,G)\right )$. In particular,
	$$
	\overline{Ch}_{I_\varphi(G)}(G)=\overline{Ch}_{I_\varphi (G)\cap\mathrm{Tran}(G)}(G)=\overline{Ch}_X(G)=h(G).
	$$
\end{corollary}

\begin{remark}
	{The previous corollary generalizes Theorem 2}
	obtained by Zhu and Ma \cite{MR4200965}. Indeed, from Lemma 3.3 of  \cite{MR4200965}, if the free semigroup action $G$ has specification property, then the skew product $F$ has specification property.  This yields from \cite{MR646049} {that there exists a $F$-invariant probability} measures on $\Sigma_{m}^+\times X$ with full support. On the other hand, from Proposition \ref{proposition almost} we know that specification implies the $\mathbf{g}$-almost product property for free semigroup actions. Therefore, the specification property of the free semigroup $G$ implies the hypothesis of Corollary \ref{full entropy 2}, as we wanted to prove. 
\end{remark}

%{Next we provide an example that satisfies the assumptions} of the Theorem \ref{full entropy 2-2}. 
%\begin{example}\label{example2}
%	Let $M$ be a compact Riemannian manifold and $G$ the free semigroup generated by $f_0,\cdots,f_{m-1}$ on $M$ which are $C^1$-local diffeomorphisms such that for any $j=0,\cdots,m-1$, $\|Df_j(x)v\|\ge\lambda_j\|v\|$ for all $x\in M$ and all $v\in T_x M$, where $\lambda_j$ is a constant larger than 1. From \cite{MR4200965} or \cite{MR3503951} and Proposition \ref{proposition almost}, one has $G$ is positively expansive with $\mathbf{g}$-almost product property.
%\end{example}

{We provide an example that satisfies the assumptions}
 of Theorem \ref{entropy of BR-1},  \ref{entropy of QW-1} and \ref{entropy of omega limit set}.
	\begin{example}\label{example 3}
		Given $q\in\mathbb{N}$, let $A_j\in\mathrm{GL}(q,\mathbb{Z})$ be the integer coefficients matrix  whose 
		{the determinant is different} from zero and eigenvalues have absolute value bigger than one, for $j=0,\cdots,m-1$. Let $f_{A_j}:\mathbb{T}^q\to\mathbb{T}^q$ be the linear endomorphism of the torus induced by the matrix $A_j$. Then the transformations $f_{A_0},f_{A_1},\cdots,f_{A_{m-1}}$ are all expanding (see Sec. 11.1 of \cite{MR3558990} for details). Let $G$ be the free semigroup action generated by $f_{A_0},f_{A_1},\cdots,f_{A_{m-1}}$. 	It follows from  \cite{MR3503951} that $G$ is positively expansive with $\mathbf{g}$-almost product property. 
		
		Suppose that $F:\Sigma_m^+\times X \to\Sigma_m^+\times X$ is the skew product map corresponding to the maps $f_{A_0},f_{A_1},\cdots,f_{A_{m-1}}$. 
		{From Section 4.2.5 of \cite{MR3558990}, we have that $f_{A_j}$ preserves the Lebesgue measure}
		$\mu$ on $\mathbb{T}^q$ for all $j=0,\cdots,m-1$. 
		Hence the Lebesgue measure is stationary with respect to any Bernoulli measure, so the skew product $F$ is not uniquely ergodic. This may be seen as follows. Let $\mathbb{P}_a$ and $\mathbb{P}_b$ be two Bernoulli measures on $\Sigma_{m}^+$ generated by different probability vectors $a=(a_0,a_1,\cdots,a_{m-1})$ and $b=(b_0, b_1,\cdots,b_{m-1})$, respectively. It follows from Proposition \ref{invariant measure of random} that the different product measures both $\mathbb{P}_a\times\mu$ and $\mathbb{P}_b\times\mu$ are invariant under the skew product $F$. This shows that $F$ is not uniquely ergodic. Hence, it satisfies the hypothesis of Theorem \ref{entropy of BR-1},  \ref{entropy of QW-1} and \ref{entropy of omega limit set}, as we wanted to prove.
	\end{example}

%--------------------------------------------------------------section 4-------------------------------------------------------------------------
\section{Proofs of the main results}\label{BR and QW}
In this section, we complete the proofs of Theorem \ref{entropy of BR-1}, \ref{entropy of QW-1} and \ref{entropy of omega limit set}.  Let $(X,f)$ be a dynamical system and $Z_{1}, Z_{2}, \cdots, Z_{k} \subseteq X(k \ge 2)$ be a collection of subsets of $X$. Recall that $\left\{Z_{i}\right\}$ has full entropy gaps with respect to $Y \subseteq X$ if
$$
h_{(Z_{i+1} \backslash Z_{i}) \cap Y}(f)=h_{Y}(f) \text { for all } 1 \le i<k.
$$

Next throughout this section we assume that $X$ is a compact metric space, $G$ is the free semigroup generated by $m$ generators $f_0,\cdots,f_{m-1}$ which are continuous maps on $X$, and $F$ is the skew product map corresponding to the maps $f_0,\cdots,f_{m-1}$.

From \cite{MR3963890}, Tian defined the recurrent level sets of the upper Banach recurrent points with respect to a single map. For the skew product map $F$, let $\mathrm{BR}^{\#}(F):=\mathrm{BR}(F) \backslash \mathrm{QW}(F)$,
$$
\begin{aligned}
	W(F) &:=\left\{(\iota,x) \in \Sigma_{m}^+\times X : S_{\mu}=C_{(\iota,x)} \text { for every } \mu \in M_{(\iota,x)}(F)\right\}, \\
	V(F) &:=\left\{(\iota, x) \in \Sigma_{m}^+\times X: \exists \mu \in M_{(\iota,x)}(F) \text { such that } S_{\mu}=C_{(\iota,x)}\right\}, \\
	S(F)&:=\left\{(\iota,x) \in \Sigma_{m}^+\times X : \cap_{\mu \in M_{(\iota,x)}(F)} S_{\mu} \neq \emptyset\right\} .
\end{aligned}
$$
More precisely,  $\mathrm{BR}^{\#}(F)$ is divided into the following several levels with different asymptotic behaviour:
$$
\begin{aligned}
	&\mathrm{BR}_{1}(F):=\mathrm{BR}^{\#}(F)\cap W(F), \\
	&\mathrm{BR}_{2}(F):=\mathrm{BR}^{\#}(F)\cap(V(F) \cap S(F)), \\
	&\mathrm{BR}_{3}(F):=\mathrm{BR}^{\#}(F)\cap V(F),\\
	&\mathrm{BR}_{4}(F):=\mathrm{BR}^{\#}(F)\cap(V(F) \cup S(F)),\\
	&\mathrm{BR}_{5}(F):=\mathrm{BR}^{\#}(F).
\end{aligned}
$$
Then $\mathrm{BR}_{1}(F) \subseteq \mathrm{BR}_{2}(F) \subseteq \mathrm{BR}_{3}(F) \subseteq \mathrm{BR}_{4}(F) \subseteq \mathrm{BR}_{5}(F)$.
\begin{proof}[Proof of Theorem \ref{entropy of BR-1}]
	Suppose $\mu$ is the $\mathbb{P}$-stationary measure with full support. Then, Proposition \ref{invariant measure of random} ensures that $\mathbb{P}\times \mu$ is an invariant measure under the skew product $F$ with support $\Sigma_{m}^+\times X$. From Lemma \ref{2g-almost product property}, the skew product $F$ has 2$\mathbf{g}$-almost product property. 

		(1) Consider a set 
		$$
		I_\alpha(F):=\left \{(\iota,x)\in\Sigma_{m}^+\times X: \inf_{\nu\in M_{(\iota,x)}(F)}\alpha(\nu)< \sup_{\nu\in M_{(\iota,x)}(F)}\alpha(\nu)\right \}.
		$$
		If $\alpha$ satisfies A.3 and $\mathrm{Int}(L_\alpha)\neq\emptyset$, it follows from Theorem 6.1(1) of \cite{MR3963890} that
		$$
			\left \{\mathrm{QR}(F), \mathrm{BR}_1(F),\mathrm{BR}_2(F),\mathrm{BR}_3(F),\mathrm{BR}_4(F),\mathrm{BR}_5(F)\right \}
		$$
		has full entropy gaps with respect to $I_\alpha(F)\cap\mathrm{Tran}(F)$. Hence, 
		$$
			h_{I_\alpha(F)\cap\mathrm{Tran}(F) \cap \left (\mathrm{BR}_1(F)\setminus \mathrm{QR}(F)\right )}(F)=h_{I_\alpha(F)\cap\mathrm{Tran}(F)}(F)=h_{\Sigma_{m}^+\times X}(F)
		$$
	and
	$$
		h_{I_\alpha(F)\cap\mathrm{Tran}(F) \cap \left (\mathrm{BR}_j(F)\setminus \mathrm{BR}_{j-1}(F)\right )}(F)=h_{I_\alpha(F)\cap\mathrm{Tran}(F)}(F)=h_{\Sigma_{m}^+\times X}(F),
	$$
		for $j=2,\cdots,5$. From the Theorem \ref{Topological entropy skew product}, we get that
		\begin{equation}\label{entropy of BR-1-1-3}
			\log m +h(G)=h_{I_\alpha(F)\cap\mathrm{Tran}(F) \cap \left (\mathrm{BR}_1(F)\setminus \mathrm{QR}(F)\right)}(F), 
		\end{equation}
		and
		\begin{equation}\label{entropy of BR-1-1-4}
			\log m +h(G)=h_{I_\alpha(F)\cap \mathrm{Tran}(F) \cap \left (\mathrm{BR}_j(F)\setminus \mathrm{BR}_{j-1}\right )}(F)\quad\text{for } j=2,\cdots,5.
		\end{equation}
		
		By the definitions of the sets, if $(\iota,x)\in I_\alpha(F)\cap\mathrm{Tran}(F)\cap \left ( \mathrm{BR}_1(F)\setminus\mathrm{QR}(F)\right )$ then 
		$$x\in I_\alpha(\iota,G)\cap\mathrm{Tran}(\iota, G)\cap \left ( \mathrm{BR}_1(\iota,G)\setminus\mathrm{QR}(\iota,G)\right ).$$ This shows that
		\begin{equation}\label{entropy of BR-1-1-1}
			I_\alpha(F)\cap\mathrm{Tran}(F) \cap \left (\mathrm{BR}_1(F)\setminus \mathrm{QR}(F)\right )\subseteq\Sigma_{m}^+\times M_1(I_\alpha,\mathrm{Tran})\subseteq\Sigma_{m}^+\times X,
		\end{equation}
		where 
		$$
		M_1(I_\alpha, \mathrm{Tran}):=\cup_{\iota\in\Sigma_{m}^+}\big \{I_\alpha(\iota,G)\cap\mathrm{Tran}(\iota, G) \cap\left (\mathrm{BR}_1(\iota, G)\setminus \mathrm{QR}(\iota, G)\right )\big \}.
		$$
		It follows using the formula (\ref{entropy of BR-1-1-1}) and Theorem \ref{Topological upper capacity entropy skew product} that
		\begin{equation}\label{entropy of BR-1-1-2}
			\begin{aligned}
				h_{\Sigma_{m}^+\times X}(F)&=h_{I_\alpha(F)\cap\mathrm{Tran}(F) \cap \left (\mathrm{BR}_1(F)\setminus \mathrm{QR}(F)\right)}(F)\\
				&\le \overline{Ch}_{\Sigma_{m}^+\times M_1(I_\alpha,\mathrm{Tran})}(F)\\
				&=\log m + \overline{Ch}_{ M_1(I_\alpha,\mathrm{Tran})}(G).
			\end{aligned}
		\end{equation}
		Combining these two relations (\ref{entropy of BR-1-1-3}) and (\ref{entropy of BR-1-1-2}), we find that
		$$
		\overline{Ch}_{ M_1(I_\alpha,\mathrm{Tran})}(G)=\overline{Ch}_X(G)=h(G).
		$$ 
		
		Denote 
		$$
		M_j (I_\alpha,\mathrm{Tran}):=\cup_{\iota\in\Sigma_{m}^+} \left \{I_\alpha(\iota,G)\cap \mathrm{Tran}(\iota, G) \cap\left (\mathrm{BR}_j(\iota, G)\setminus \mathrm{BR}_{j-1}(\iota, G)\right )\right \},
		$$ for $j=2,\cdots,5.$
		Similarly, if $(\iota,x)\in I_\alpha(F)\cap\mathrm{Tran}(F)\cap \left ( \mathrm{BR}_j(F)\setminus\mathrm{BR}_{j-1}(F)\right )$, we obtain that  $x\in I_\alpha(\iota,G)\cap\mathrm{Tran}(\iota, G)\cap \left ( \mathrm{BR}_j(\iota,G)\setminus\mathrm{BR}_{j-1}(\iota,G)\right )$. This shows that
		$$
		I_\alpha(F)\cap\mathrm{Tran}(F) \cap \left (\mathrm{BR}_j(F)\setminus \mathrm{BR}_{j-1}(F)\right )\subseteq\Sigma_{m}^+\times M_j(I_\alpha,\mathrm{Tran})\subseteq\Sigma_{m}^+\times X.
		$$
		It follows using the Theorem \ref{Topological upper capacity entropy skew product} that
		\begin{equation}\label{entropy of BR-1-1-5}
		\begin{aligned}
			h_{\Sigma_{m}^+\times X}(F)
			&=h_{I_\alpha(F)\cap\mathrm{Tran}(F) \cap \left (\mathrm{BR}_j(F)\setminus \mathrm{BR}_{j-1}(F)\right)}(F)\\
			&\le \overline{Ch}_{\Sigma_{m}^+\times M_j(I_\alpha,\mathrm{Tran})}(F)\\
			&=\log m + \overline{Ch}_{ M_j(I_\alpha,\mathrm{Tran})}(G).
		\end{aligned}
		\end{equation}
		Combining these two relations (\ref{entropy of BR-1-1-3}) and (\ref{entropy of BR-1-1-5}), we find that
		$$
		\overline{Ch}_{ M_j(I_\alpha, \mathrm{Tran})}(G) =\overline{Ch}_X(G)=h(G),\,\text{for }j=2,\cdots,5.
		$$ 
		This completes the proof.
		
		(2) By the hypothesis about the skew product, it follows from Theorem 6.1(3) of \cite{MR3963890} that
		$$
			\left \{\emptyset,\mathrm{QR}(F)\cap \mathrm{BR}_1(F), \mathrm{BR}_1(F),\mathrm{BR}_2(F),\mathrm{BR}_3(F),\mathrm{BR}_4(F),\mathrm{BR}_5(F)\right \}
			$$
		has full entropy gaps with respect to $\mathrm{Tran}(F)$. Hence,
		$$
		h_{\mathrm{Tran}(F)\cap \mathrm{QR}(F)\cap \mathrm{BR}_1(F)}(F)=h_{\mathrm{Tran}(F)}(F)=h_{\Sigma_{m}^+\times X}(F),
		$$
		$$
		h_{\mathrm{Tran}(F) \cap \left (\mathrm{BR}_1(F)\setminus \mathrm{QR}(F)\right )}(F)=h_{\mathrm{Tran}(F)}(F)=h_{\Sigma_{m}^+\times X}(F),
		$$
		$$
		h_{\mathrm{Tran}(F) \cap \left (\mathrm{BR}_j(F)\setminus \mathrm{BR}_{j-1}(F)\right )}(F)=h_{\mathrm{Tran}(F)}(F)=h_{\Sigma_{m}^+\times X}(F),\,\text{for } j=2,\cdots,5.
		$$
		Applying Theorem \ref{Topological entropy skew product}, we obtain that
		\begin{equation}\label{entropy of BR-1-2-4}
				\log m +h(G)=h_{\mathrm{Tran}(F)\cap \mathrm{QR}(F)\cap \mathrm{BR}_1(F)}(F),
		\end{equation}
		\begin{equation}\label{entropy of BR-1-2-2}
			\log m +h(G)=h_{\mathrm{Tran}(F) \cap \left (\mathrm{BR}_1(F)\setminus \mathrm{QR}(F)\right)}(F),
		\end{equation}
		\begin{equation}\label{entropy of BR-1-2-3}
			\log m +h(G)=h_{\mathrm{Tran}(F) \cap \left (\mathrm{BR}_j(F)\setminus \mathrm{BR}_{j-1}(F)\right )}(F),\,\text{for } j=2,\cdots,5.
		\end{equation}
	
		By the definitions of the sets, if $(\iota,x)\in\mathrm{Tran}(F)\cap \mathrm{QR}(F)\cap \mathrm{BR}_1(F)$, then $x\in\mathrm{Tran}(\iota, G)\cap \mathrm{QR}(\iota,G)\cap \mathrm{BR}_1(\iota,G)$. This shows that
		$$
		\mathrm{Tran}(F)\cap \mathrm{QR}(F)\cap \mathrm{BR}_1(F)\subseteq\Sigma_{m}^+\times M_0(\mathrm{Tran})\subseteq\Sigma_{m}^+\times X,
		$$
		where $M_0(\mathrm{Tran}):=\bigcup_{\iota\in\Sigma_{m}^+}\{\mathrm{Tran}(\iota, G)\cap \mathrm{QR}(\iota,G)\cap \mathrm{BR}_1(\iota,G)\}$. 
		It follows using the Theorem \ref{Topological upper capacity entropy skew product} that
		\begin{equation}\label{entropy of BR-1-2-5}
			\begin{aligned}
				h_{\Sigma_{m}^+\times X}(F)
				&=h_{\mathrm{Tran}(F)\cap \mathrm{QR}(F)\cap \mathrm{BR}_1(F)}(F)\\
				&\le \overline{Ch}_{\Sigma_{m}^+\times M_0(\mathrm{Tran})}(F)\\
				&=\log m + \overline{Ch}_{ M_0(\mathrm{Tran})}(G).
			\end{aligned}
		\end{equation}
		Combining these two relations (\ref{entropy of BR-1-2-4}) and (\ref{entropy of BR-1-2-5}), we find that
		$$
		\overline{Ch}_{ M_0(\mathrm{Tran})}(G)=\overline{Ch}_X(G)=h(G).
		$$ 
		
		By the definitions of the sets, if $(\iota,x)\in\mathrm{Tran}(F)\cap \left ( \mathrm{BR}_1(F)\setminus\mathrm{QR}(F)\right )$, then $x\in\mathrm{Tran}(\iota, G)\cap \left ( \mathrm{BR}_1(\iota,G)\setminus\mathrm{QR}(\iota,G)\right )$. This shows that
		$$
		\mathrm{Tran}(F) \cap \left (\mathrm{BR}_1(F)\setminus \mathrm{QR}(F)\right )\subseteq\Sigma_{m}^+\times M_1(\mathrm{Tran})\subseteq\Sigma_{m}^+\times X,
		$$
		where $M_1(\mathrm{Tran}):=\bigcup_{\iota\in\Sigma_{m}^+} \left \{\mathrm{Tran}(\iota, G) \cap\left (\mathrm{BR}_1(\iota, G)\setminus \mathrm{QR}(\iota, G)\right )\right \}$. It follows using the Theorem \ref{Topological upper capacity entropy skew product} that
		\begin{equation}\label{entropy of BR-1-2-6}
		\begin{aligned}
			h_{\Sigma_{m}^+\times X}(F)
			&=h_{\mathrm{Tran}(F) \cap \left (\mathrm{BR}_1(F)\setminus \mathrm{QR}(F)\right)}(F)\\
			&\le \overline{Ch}_{\Sigma_{m}^+\times M_1(\mathrm{Tran})}(F)\\
			&=\log m + \overline{Ch}_{ M_1(\mathrm{Tran})}(G).
		\end{aligned}
		\end{equation}
		Combining these two relations (\ref{entropy of BR-1-2-2}) and (\ref{entropy of BR-1-2-6}), we find that
		$$
		\overline{Ch}_{ M_1(\mathrm{Tran})}(G)=\overline{Ch}_X(G)=h(G).
		$$ 
		
		Denote 
		$$
		M_j (\mathrm{Tran}):=\cup _{\iota\in\Sigma_{m}^+} \left \{\mathrm{Tran}(\iota, G) \cap\left (\mathrm{BR}_j(\iota, G)\setminus \mathrm{BR}_{j-1}(\iota, G)\right )\right \}, 
		$$
		for $j=2,\cdots,5$.
		Similarly, if $(\iota,x)\in\mathrm{Tran}(F)\cap \left ( \mathrm{BR}_j(F)\setminus\mathrm{BR}_{j-1}(F)\right )$, we obtain that  $x\in\mathrm{Tran}(\iota, G)\cap \left ( \mathrm{BR}_j(\iota,G)\setminus\mathrm{BR}_{j-1}(\iota,G)\right )$. This shows that
		$$
		\mathrm{Tran}(F) \cap \left (\mathrm{BR}_j(F)\setminus \mathrm{BR}_{j-1}(F)\right )\subseteq\Sigma_{m}^+\times M_j(\mathrm{Tran})\subseteq\Sigma_{m}^+\times X.
		$$
		Analogously,  using Theorem \ref{Topological upper capacity entropy skew product} and the formula \ref{entropy of BR-1-2-3}, we conclude that
		$$
		\overline{Ch}_{ M_j(\mathrm{Tran})}(G)=\overline{Ch}_X(G)=h(G), \,\text{for }j=2,\cdots,5.
		$$ 	
		This completes the proof.
		
		(3) Define a set as
		$$
		R_{\alpha}(F):=\left \{(\iota,x)\in\Sigma_{m}^+\times X: \inf_{\nu\in M_{(\iota,x)}(F)}\alpha(\nu)=\sup_{\nu\in M_{(\iota,x)}(F)}\alpha(\nu)\right \}. 
		$$
		If the skew product $F$ is not uniquely ergodic and $\alpha$ satisfies A.1 and A.2, from Theorem 6.1(4) of \cite{MR3963890}, we obtain that
		\begin{equation}\label{entropy of BR-1-2-1}
			\left \{\emptyset,\mathrm{QR}(F)\cap \mathrm{BR}_1(F), \mathrm{BR}_1(F),\mathrm{BR}_2(F),\mathrm{BR}_3(F),\mathrm{BR}_4(F),\mathrm{BR}_5(F)\right \}
		\end{equation}
		has full entropy gaps with respect to $R_\alpha(F)\cap\mathrm{Tran}(F)$. Similar to Theorem \ref{entropy of BR-1} (1) and (2), one can adopt the proof to complete the proof for $R_\alpha (\iota,G)\cap \mathrm{Tran}(\iota,G)$. Here we omit the details.
\end{proof}

\begin{proof}[Proof of Theorem \ref{entropy of QW-1}]
	By Lemma 3.4 of \cite{MR4200965}, it follows that $G$ is expansive if and only if the skew product $F$ is expansive. In the same spirit, one can adapt the proof of Theorem \ref{entropy of BR-1} using Theorem 7.1 of \cite{MR3963890} to complete the proof. Here we omit the details.
\end{proof}
\begin{remark}
	We extend the some results of  \cite{MR3963890} to the dynamical systems of free semigroup actions.
\end{remark}
\begin{corollary}\label{entropy of BR}
	Suppose that $G$ has the $\mathbf{g}$-almost product property, there exists a $\mathbb{P}$-stationary measure on $X$ with full support where $\mathbb{P}$ is a Bernoulli measure on $\Sigma_{m}^+$.  Let $\varphi: X\to\mathbb{R}$ be a continuous function.
	\begin{itemize}
		\item [(1)] If $I_\varphi (\iota,G)\neq \emptyset$ for some $\iota\in\Sigma_{m}^+$, then the unions of gaps of 
		$$
		\left \{\mathrm{QR}(\iota, G), \mathrm{BR}_1(\iota,G), \mathrm{BR}_2(\iota,G),\mathrm{BR}_3(\iota,G),\mathrm{BR}_4(\iota,G),\mathrm{BR}_5(\iota,G)\right \}
		$$
		with respect to $I_\varphi (\iota,G)\cap\mathrm{Tran}(\iota,G)$ for all $\iota\in\Sigma_m^+$ have full upper capacity topological entropy of free semigroup action $G$.
		\item [(2)] If the skew product $F$ is not uniquely ergodic, then the unions of gaps of 
		$$
		\left \{\emptyset, \mathrm{QR}(\iota, G)\cap \mathrm{BR}_1(\iota,G),\mathrm{BR}_1(\iota,G), \mathrm{BR}_2(\iota,G),\mathrm{BR}_3(\iota,G),\mathrm{BR}_4(\iota,G),\mathrm{BR}_5(\iota,G)\right \}
		$$
		with respect to $\mathrm{Tran}(\iota,G)$, $R_\varphi(\iota,G)\cap\mathrm{Tran}(\iota,G)$ for all $\iota\in\Sigma_{m}^+$ have full upper capacity topological entropy of free semigroup action $G$, respectively.
	\end{itemize}
\end{corollary}
\begin{corollary}\label{entropy of QW}
	Suppose that $G$ has the $\mathbf{g}$-almost product property and positively expansive, there exists a $\mathbb{P}$-stationary measure with full support on $X$ where $\mathbb{P}$ is a Bernoulli measure on $\Sigma_{m}^+$. Let $\varphi: X\to\mathbb{R}$ be a continuous function. If the skew product $F$ is not uniquely ergodic, then the unions of gaps of 
	$$
	\left \{\emptyset, \mathrm{QW}_1(\iota,G),\mathrm{QW}_2(\iota,G),\mathrm{QW}_3(\iota,G),\mathrm{QW}_4(\iota,G),\mathrm{QW}_5(\iota,G)\right \}
	$$
	 with respect to 
	$\mathrm{Tran}(\iota,G)$, $R_\varphi(\iota,G)\cap\mathrm{Tran}(\iota,G)$ for all $\iota\in\Sigma^+_m$ have full upper capacity topological entropy of free semigroup action $G$, respectively. {If $I_\varphi(\iota,G)$}
	is non-empty for some $\iota\in\Sigma_{m}^+$, similar arguments hold with respect to $I_\varphi (\iota,G)\cap\mathrm{Tran}(\iota,G)$.
\end{corollary}

\begin{proof}[Proof of {Theorem} \ref{entropy of omega limit set}]
	Suppose that $\mu$ is the $\mathbb{P}$-stationary measure with full support. Then, Proposition \ref{invariant measure of random} ensures that $\mathbb{P}\times \mu$ is an invariant measure under the skew product transformation $F$ with support $\Sigma_{m}^+\times X$. From Lemma \ref{2g-almost product property}, the skew product $F$ has 2$\mathbf{g}$-almost product property. For $j=1,\cdots,6$, let us denote:
	$$
	T_j(F):=\left \{(\iota,x)\in \mathrm{Tran} (F):(\iota,x) \text{ satisfies Case } (j)\right \},
	$$
	and 
	$$
	B_j(F):=\left \{(\iota,x)\in \mathrm{BR} (F):(\iota,x) \text{ satisfies Case } (j)\right \}.
	$$
	It follows from Theorem 1.3 of \cite{MR3963890} that
	$$
	T_j(F)\neq\emptyset,\quad B_j(F)\neq\emptyset,
	$$
	and they all have full topological entropy for all $j=1,\cdots,6$.

	For $j=1,\cdots,6$, if $(\iota,x)\in T_j(F)$ with $\iota=(i_0,i_1,\cdots)$, then the orbit of $(\iota,x)$ under $F$, that is, $\{F^k(\iota,x):k\in\mathbb{N} \}$ is dense in $\Sigma_{m}^+\times X$. This implies that $orb(x,\iota,G)$ is dense in $X$, hence $x\in \mathrm{Tran}(\iota,G)$. Notice that $(\iota,x)$ satisfies Case $(j)$. Hence $x\in T_j(G)$ and $T_j(G)\neq\emptyset$. In particular, one has that 
	$$
	T_j(F)\subseteq\Sigma_{m}^+\times T_j(G) \subseteq\Sigma_{m}^+\times X.
	$$
	Summing the two Theorems \ref{Topological entropy skew product} and \ref{Topological upper capacity entropy skew product}, we get that
	$$
	\begin{aligned}
		\log m + h(G)=h_{\Sigma_{m}^+\times X}(F)
		&=h_{T_j(F)}(F)\\
		&\le \overline{Ch}_{\Sigma_{m}^+\times T_j(G)}(F)\\
		&=\log m +\overline{Ch}_{T_j(G)}(G).
	\end{aligned}
	$$
	Since $h(G)=\overline{Ch}_X(G)$, this proves that
	$$
	h(G)=\overline{Ch}_X(G)\le\overline{Ch}_{T_j(G)}(G).
	$$
	Finally, we conclude that
	$$
	h(G)=\overline{Ch}_X(G)=\overline{Ch}_{T_j(G)}(G).
	$$

	On the other hand, 	for $j=1,\cdots,6$, if $(\iota,x)\in B_j(F)$ with $\iota=(i_0,i_1,\cdots)$, then for any $\varepsilon>0$, the set of visiting time $N\left ((\iota,x), B\left ((\iota,x),\varepsilon\right )\right )$ has a positive Banach upper density and so does $N_\iota(x,B(x,\varepsilon))$. Hence, we get that $x\in \mathrm{BR}(\iota,G)$. Using the fact that $(\iota,x)$ satisfies Case $(j)$,  hence $x\in B_j(G)$, so $B_j(G)\neq\emptyset$. In particular, one has that 
	$$
	B_j(F)\subseteq\Sigma_{m}^+\times B_j(G) \subseteq\Sigma_{m}^+\times X.
	$$
	Summing the two Theorems \ref{Topological entropy skew product} and \ref{Topological upper capacity entropy skew product}, we get that
	$$
	\begin{aligned}
		\log m + h(G)=h_{\Sigma_{m}^+\times X}(F)
		&=h_{B_j(F)}(F)\\
		&\le \overline{Ch}_{\Sigma_{m}^+\times B_j(G)}(F)\\
		&=\log m +\overline{Ch}_{B_j(G)}(G).
	\end{aligned}
	$$
	Since $h(G)=\overline{Ch}_X(G)$, this proves that
	$$
	h(G)=\overline{Ch}_X(G)\le\overline{Ch}_{B_j(G)}(G).
	$$
	Finally, we conclude that
	$$
	h(G)=\overline{Ch}_X(G)=\overline{Ch}_{B_j(G)}(G).
	$$
\end{proof}
\begin{remark}
	Theorem \ref{entropy of omega limit set} is a generalization of Theorem 1.3 of \cite{MR3963890}.
\end{remark}

\textbf{Acknowledgements}
The authors really appreciate the referees’ valuable remarks and suggestions that helped a lot. 

\bibliography{sn-bibliography}% common bib file
\end{document}